\documentclass[a4paper, 10 pt, reqno]{amsart}
\usepackage{amsfonts, bbm}
\usepackage[psamsfonts]{amssymb}
\usepackage{enumerate}
\usepackage[dvips]{graphicx} 
\usepackage[all, cmtip]{xy} 
\usepackage{url} 
\usepackage[applemac]{inputenc}
\usepackage[T1]{fontenc}
\usepackage{xspace}
\usepackage{lmodern}
\usepackage{textcomp}
\theoremstyle{plain}
\newtheorem{thm}{Theorem}[section]

\newtheorem{prp}[thm]{Proposition}

\newtheorem{lem}[thm]{Lemma}
\newtheorem{crl}[thm]{Corollary}

\newtheorem{problem}[thm]{Problem}
\newtheorem{question}[thm]{Question}
\newtheorem{questions}[thm]{Questions}

\theoremstyle{definition}
\newtheorem{definition}[thm]{Definition}

\theoremstyle{remark}
\newtheorem{remark}[thm]{Remark}
\newtheorem{remarks}[thm]{Remarks}
\newtheorem{example}[thm]{Example}
\newtheorem{examples}[thm]{Examples}
\newtheorem{illustration}[thm]{Illustration}

\input{Definitions.tex}
\title{Metabelianisations of finitely presented groups}
\author{Ralph Strebel}
\date{5nd of November 2018, version v9}
\begin{document}
\thispagestyle{plain}
%
\begin{abstract}
In their paper \cite{PaSu17}, S. Papadima and A. Suciu write
at the end of Subsection 1.1:
\begin{quote}
Further motivation for considering
Question 1.1 comes from an effort to understand 
whether the maximal metabelian quotient $G = \Gamma/\Gamma''$ of a finitely presented group $\Gamma$ is also finitely presentable.
\end{quote}

\noindent
As an answer to the stated problem,
the authors investigate the class of finitely presented, very large groups   
$\Gamma$ and show 
that metabelian top $\Gamma/\Gamma''$ of every group in this class is infinitely related.
In this article, 
I study some further classes of finitely presented groups
with the aim of finding out 
whether or not the metabelian tops of the members of these classes admit finite presentations.
\end{abstract}
\maketitle 

\section{Introduction}
\label{sec:Introduction}
Suppose $\Gamma$ is a group admitting a finite presentation
and $G$ is a quotient group of $\Gamma$.
Then $G$ is, of course, finitely generated but it need not be finitely presentable.
This fact is well-known 
and follows, for instance, from B. H. Neumann's observation 
that there are only countably many groups with a finite presentation,
but that the cardinality of pairwise non-isomorphic finitely generated groups is larger, namely $2^{\aleph_0}$ (\cite[Theorem 14]{Neu37}).
\footnote{As pointed out by P. Hall on page 433 of \cite{Hal54},
the cardinality of the class of two generated, center-by-metabelian groups
is likewise $2^{\aleph_0}$.}
Of course, there are also specific examples of finitely generated groups 
which are known to be infinitely related, 
for instance the wreath product $G = \Z \wr C_\infty$ of two infinite cyclic groups (see, \eg Theorem 1 in \cite{Bau61}).

In their recent paper \cite{PaSu17},
S. Papadima and A. Sugiu consider couples of finitely generated groups
\[
(\Gamma, G= \Gamma/\Gamma'')
\]
 and show, in Theorem 1.3,  
 that the quotient group $G$ is infinitely related
 whenever $\Gamma$ is a finitely presentable \emph{very large group}.
 \footnote{A finitely generated group is called \emph{very large}  
 if it admits a non-abelian free quotient group.}
 
 In this article, 
 I follow the lead of Papadima and Sugiu 
 and study familiar classes of finitely presented groups $\Gamma$ 
 with the aim of finding out
 whether or not their metabelian tops $G = \Gamma/\Gamma''$
are finitely related.

Actually, three kinds of classes will be considered:
\begin{enumerate}[(I)]
\item for the classes of the first kind,
the metabelian tops of all groups in the class are finitely presented;
\item for the classes of the second kind,
the metabelian top of the groups in the class can be finitely related 
and whether this is the case depends on a known parameter in a simple way;
\item the metabelian tops of the groups in the classes of the third kind 
exhibit the same behaviour as the groups studied by Papadima and Sugiu: 
none of them does admit a finite presentation.
\end{enumerate}

Here is a brief description of the layout of the paper.
Section \ref{sec:Preliminaries} 
collects known results about metabelian groups
that are needed for an understanding of the main part of the paper.
This main part begins with Section \ref{sec:Three-basic-results};
in it three results are stated and proved 
that form the basis of the investigations undertaken in later sections.
A first class, 
that of finitely presented soluble group,
is the theme of Section
\ref{sec:Fp-groups-with-no-non-abelian-free-subgroups}; 
this class is of kind (I).
In Sections \ref{sec:One-relator-groups} through 
\ref{sec:Artin-systems-of-finite-type},
one-relator groups, knot groups, Artin groups 
and, in more detail, Artin groups of finite type, are studied;
these classes are of the second kind.
The paper ends with classes of type (III); 
they are discussed in Section 
\ref{sec:Classes-of-groups-with-infinitely-related-metabelian-tops}.

Each of the Sections 
\ref{sec:Three-basic-results}
through
\ref{sec:Classes-of-groups-with-infinitely-related-metabelian-tops}
begins with a brief introduction putting its topic into context; 
so there is no need here to present these classes in more detail.
\smallskip

\emph{Acknowledgements.}
I thank Markus Brodmann and Ian J. Leary for helpful discussions and suggestions.

\section{Preliminaries}
\label{sec:Preliminaries}

\subsection{Some properties of finitely generated metabelian groups}
\label{ssec:Properties-of-metabelian-groups}
In this section, 
I collect some properties and remarks about finitely generated metabelian groups
that will be helpful in later sections.
\smallskip

Since many of the results of this paper deal with finitely presented groups
it is appropriate to begin the list of properties of metabelian groups 
with an old observation of B. H. Neumann's 
(see \cite[(8) Lemma]{Neu37} or \cite[2.2.3]{Rob96}):
\begin{prp}
\label{prp:Neumans-result}
If the kernel of an epimorphism $F \epi Q$ of a finitely generated free $F$  onto a group $Q$ is the normal closure of a finite set,
so is the kernel of every epimorphism of a finitely generated group 
onto $Q$.
\end{prp}

Now to finitely generated metabelian groups.
By definition, a group $G$ is \emph{met\-abelian} 
if its second derived group $G'' = ((G')'$ 
is reduced to the unit element of $G$.
Let $G$ be a \emph{finitely generated metabelian} group.
Its abelianisation $G_{\ab}$ is then, of course, finitely generated;
in addition, $G_{\ab}$ admits a finite presentation
(this follows, \eg from the fact 
that every cyclic group admits a presentation with one generator and one relator,
and from \cite[2.2.4]{Rob96}).
The derived group $G'$ of $G$ is therefore finitely generated 
as a normal subgroup (\cite[Thm.\;3]{Hal54}).
Now $G'$ is abelian; so it is often written additively.
If this is done,
$G'$ becomes a module $A$ over $\Z{G_{\ab}}$.
The fact that $G'$, as a normal subgroup of $G$, 
is finitely generated translates into the statement 
that the $\Z{G_{\ab}}$-module $A$ is finitely generated over $\Z{G_{\ab}}$.

I move on to another result about metabelian groups $G$
that will be exploited repeatedly in the sequel,
namely the fact that 
\emph{every finitely generated metabelian  group satisfies the maximal condition
on normal subgroups}
(\cite{Hal54}, \cf\cite[15.3.1]{Rob96}).
Put differently,
if $G$ a finitely generated metabelian group 
then each of its normal subgroups is the normal closure of a finite set of elements.

Prior to exploiting the stated result 
I pause for an observation.
Assume $\Gamma$ is a finitely presented group.
By definition, 
there exists then a finitely generated free group $F$ 
and an epimorphism $\pi \colon F \epi \Gamma$ such 
that $R = \ker \pi$ is the normal closure of finitely many relators.
\footnote{By Neumann's Lemma  \ref{prp:Neumans-result}
the choices of $F$ and of $\pi$ play no rôle as long as $F$ is of finite rank
and as long as one is not interested in the minimal number of relators.}
Consider now a normal subgroup $N$ of $\Gamma$ 
that can be generated, as a normal subgroup,  by finitely many elements.
Let $\rho \colon G \epi \Gamma/N$ be the canonical projection
and consider the composition 
\[
\rho \circ \pi \colon F \epi \Gamma \epi \Gamma/N.
\]
The kernel of $\rho \circ\pi$ is then finitely generated as normal subgroup of $F$
and so the group $\Gamma/N$ is finitely presentable.

In the sequel,
a variant of the contraposition of the implication just proved 
will be invoked repeatedly, namely 
\begin{lem}
\label{lem:Infinitely-related-metabelian-top}
Suppose $\Gamma$ is a finitely generated group
which has a metabelian image $\Gamma/N$ 
that does \emph{not} admit a finite presentation.
Then the metabelian top $\Gamma/\Gamma''$ of $\Gamma$ is infinitely related, too.
\end{lem} 

\begin{proof}
The normal subgroup $N$ of $\Gamma$ 
contains the second derived group $\Gamma''$ of $\Gamma$.
Since  the quotient $N/\Gamma''$ is a normal subgroup of $\Gamma/\Gamma''$,
it is the normal closure of finitely many elements by P. Hall's result mentioned before.
\end{proof}
\subsection{Sigma-Theory for metabelian groups}
\label{ssec:Review-Sigma-theory}
The Sigma-Theory for finitely generated \emph{metabelian} groups goes back 
to the paper
\cite{BiSt80} by Bieri and Strebel.
In \cite{BiSt81a} this theory was developed further.
More important then was the new approach to the theory
propounded by Bieri and Groves in \cite{BiGr84}.

I recall now the basic definitions introduced in \cite{BiSt80} 
and formulate a result that will be a corner stone of many results of this article. 
I shall use, however, the newer terminology going back to 
\cite[Section 8]{BiGr84} and so I shall say \emph{character} instead of valuation
and \emph{character sphere} instead of valuation sphere.

Let $Q$ be a finitely generated, multiplicatively written abelian group.
A \emph{character} of $Q$ is a homomorphism $\chi \colon Q \to \R$ of $Q$ 
into the additive group of the field of real numbers $\R$.
Two characters $\chi$, $\chi'$ of $Q$ are called \emph{equivalent} 
if one is a positive multiple of the other.
The relation so defined is an equivalence relation;
the class represented by the character $\chi$ will be denoted by $[\chi]$.

Since $Q$ is finitely generated the real vector space 
$V =\Hom(Q, \R)$ of all characters of $Q$ is a finite dimensional, real vector space whose dimension equals the torsion-free abelian rank of $Q$.
The topologies induced on $\Hom(Q, \R)$ by the norms on $V$ coincide therefore 
and give $\Hom(Q, \R)$ the structure of a well-defined topological vector space.
The classes of equivalent characters correspond to closed rays in $V$ 
that emanate from the origin $0$.
 If one omits the origin and equips the set of equivalence classes of non-zero characters with the quotient topology induced from $V \smallsetminus \{0\}$ 
one arrives at a topological space $S(Q)$.
This space is homeomorphic to the unit sphere $\s^{\dim V - 1}$
 in euclidean space of dimension $\dim V$
and it is called the \emph{character sphere} of $Q$.

Assume next that $A$ is a finitely generated (left) $\Z{Q}$-module.
Associate to it a subset $\Sigma_A(Q)$ of $S(Q)$ as follows.
Given a non-zero character $\chi$ of $S(Q)$,
define
\begin{equation}
\label{eq:Character-moniod}
Q_\chi = \{ [\chi]\in Q \mid \chi(q) \geq 0\}.
\end{equation}
This subset is a submonoid of the multiplicative group $Q$ 
and it does not change if $\chi$ is replaced by a positive multiple of it.
One can therefore associate to the module $A$ a well-defined subset of $S(Q)$
by putting
\begin{equation}
\label{eq:Sigma-A}
\Sigma_A(Q) 
= 
\{ [\chi] \in S(Q) \mid A \text{ is finitely generated over the monoid ring } \Z{Q_\chi} \}.
\end{equation}
This subset is called the \emph{geometric invariant} of the module $A$.

The set $\Sigma_A(Q)$ is hard to work with directly,
but there are alternative expressions for it.
In the sequel I use only one such expression;
it is related to the centralizer
\[
C(A) = \{\lambda \in \Z{Q} \mid \lambda \cdot a = a   \text{ for all } a \in A \}
\]
of $A$ in $\Z{Q}$.

\begin{lem}[Sigma-Criterion]
\label{lem:Sigma-criterion}
Let $A$ be a finitely generated $\Z{Q}$-module 
and let $\chi \colon Q \to \R$ be a non-zero character.
Then $A$ is finitely generated over $\Z{Q_\chi}$ if, and only if, 
there exists  an element  $\lambda = \sum_q c_q \cdot q$  in $C(A)$ 
such that $\chi(q) > 0$ for every $q$ in the support of $\lambda$.
Moreover, if $A$ is finitely generated over $\Z{Q_\chi}$,
every set generating $A$ as a $\Z{Q}$-module 
also generates $A$ as a $\Z{Q}$-module.
\end{lem}

\begin{proof}
See Proposition 2.1 in \cite{BiSt80} or Lemma 19 in \cite{Str84}.
\end{proof}

The Sigma-criterion has a very useful and important consequence:
it implies 
that $\Sigma_A(Q)$ is always an open subset of the sphere $S(Q)$ 
(see, \cite[Prop.\;2.2(i)]{BiSt80}). 
\smallskip

I recall, finally, a main result of the article \cite{BiSt80}.
\begin{thm}[\protect{\cite[Thm.\;A]{BiSt80}}]
\label{thm:Theorem-A}
Let $G$ be a finitely generated metabelian group 
that is an extension of an abelian, normal subgroup $A$ by an abelian group $Q$ and equip $A$ with the structure of left $\Z{Q}$-module induced by conjugation.
Then $A$ is a finitely generated $\Z{Q}$-module and the following statements hold:
\begin{enumerate}[(i)]
\item $G$ is polycyclic if, and only if, $\Sigma_A(Q) = S(Q)$.
\item $G$ admits a finite presentation if, and only if, 
$\Sigma_A(Q) \cup - \Sigma_A(Q) = S(Q)$.
\end{enumerate}
In the above, $-\Sigma_A(Q)$ 
denotes the set of antipodes of the points in $\Sigma_A(Q)$.
\end{thm}

%
%
%
%
%
%

\section{Three basic results}
\label{sec:Three-basic-results}
In this section, I establish three results which help one in determining
whether or not the metabelian top $G = \Gamma/\Gamma''$ of a finitely presented group $\Gamma$ admits a finite presentation.
The first result describes a sufficient, as well as necessary, 
condition for $G$ to admit a finite presentation
when $G$ is a finitely generated metabelian group of a special kind;
the second result states a sufficient condition for $G$ to be finitely related,
while the third result lists various conditions 
which force $G$ to be infinitely related.
Applications of each of these three results will be given in later sections.
\subsection{First result}
\label{ssec:First-basic-result}
The first result deals with finitely generated metabelian groups
and is a rather simple consequence of the characterization of finitely presented metabelian groups propounded in \cite{BiSt80}. 
\begin{thm}
\label{thm:FR-implies-polycyclic}
Let $G$ be a finitely generated metabelian group 
which admits an automorphism $\iota$ 
that induces the automorphism $gG' \mapsto g^{-1}G'$ 
on the abelianisation $G_{\ab}$ of $G$.
Then $G$ is finitely related if, and only if it is polycyclic.
\end{thm}

\begin{proof}
Set $Q = G/G'$ and let $A$ denote the abelian group $G'$ 
equipped with the $\Z{Q}$-module structure induced by conjugation.
The automorphism $\iota$ of $G$ induces an automorphism $\iota_{\ab}$
on the abelianisation $Q$ of $G$.
This automorphism implies  
that the subset $\Sigma_A(Q)$ is invariant under the antipodal map (see below).

Theorem \ref{thm:Theorem-A} then allows one to argue as follows.
By statement (ii) of the theorem,
the group $G$ is finitely presented if, and only if, 
$\Sigma_A(Q) \cup -\Sigma_A(Q) = S(Q)$.
By the stated invariance of $\Sigma_A(Q)$,
the equality $\Sigma_A(Q) \cup -\Sigma_A(Q) = S(Q)$ holds if, and only if
$\Sigma_A(Q) = S(Q)$, 
and by the statement (i) of the Theorem \ref{thm:Theorem-A} 
this last condition is fulfilled precisely if $G$ is polycyclic.

We are left with proving that $\Sigma_A(Q) = -\Sigma_A(Q)$.
Let $\chi \colon Q \to \R$ be a non-zero character of $Q$
and assume that $[\chi] \in \Sigma_A(Q)$.
Lemma \ref{lem:Sigma-criterion} then guarantees
that there exists a group ring element 
$\lambda = \sum_k c_k  \cdot q_k$ in  $C(A)$ 
so that  $a = \lambda \cdot a$ for every $a \in A$ 
and so that $\chi(q_k) > 0$ for $q_k \in \supp \lambda$.
Choose, for every $q_k$ an element $g_k \in G$ with $q_k = g_k G'$.
Every $a \in A$ can then be written in the form
\[
a=
\left( \sum\nolimits _k c_k \cdot q_k \right) \cdot a 
=
\sum\nolimits _k  c_k \act{g_k}a.
\]
By applying  the automorphism $\iota$ to this representation,
one gets
\begin{align}
\iota(a) 
&=
\iota \left( \sum\nolimits_k  c_k \act{g_k}a \right)
=
\sum\nolimits_k c_k \cdot \act{\iota(g_k)} \iota (a) 
=
\left(\sum\nolimits_k c_k \cdot \iota_{\ab}{(q_k)}\right)  \cdot \iota (a) 
\label{eq:Representation-of-iota(a)}\\
&= 
\left(\sum\nolimits_k c_k \cdot q_k^{-1} \right) \cdot \iota(a).
\notag
\end{align}
Since $\iota \colon G \to G$ is bijective, 
so is its restriction to $A = G'$;
calculation \eqref{eq:Representation-of-iota(a)}
thus proves
that the group ring element $\sum\nolimits_k c_k \cdot q_k^{-1}$ belongs to the centralizer of $A$.
As $(-\chi) (q_k^{-1}) = \chi (q_k) > 0$, 
Lemma \ref{lem:Sigma-criterion} allows one to infer
that $[-\chi] \in \Sigma_A(Q)$.

All taken together, 
the previous calculation proves the implication 
\[
[\chi] \in \Sigma_A(Q) \Rightarrow  [-\chi] \in \Sigma_A(Q).
\]
Its converse holds by symmetry.
\end{proof}

\begin{remark}
\label{remark:Dichotomy}
Theorem \ref{thm:FR-implies-polycyclic} 
will be applied to Artin systems in Section \ref{sec:Artin-systems}.
The standard presentation of an Artin system $(\Gamma, S)$
allows one to deduce with ease
that the map which sends each generator $s$
of the standard generating set $S$ to its inverse $s^{-1}$
extends to an automorphism of $\Gamma$.
The metabelian top $G = \Gamma/\Gamma''$ 
will therefore admit an automorphism $\iota$ 
as described in the statement of Theorem \ref{thm:FR-implies-polycyclic}.
\end{remark}
%
%
\subsection{Second result}
\label{ssec:Second-basic-result}
%
An essential hypothesis of Theorem \ref{thm:FR-implies-polycyclic} is the requirement that the metabelian group $G$ admits an automorphism
which induces $-\id$ on  $G_{\ab}$.
There are, however, no restrictions 
on the structure of the abelianized group.

The second result makes a strong restriction on $G_{\ab}$:
it demands that it be infinite cyclic;
moreover,
it requires 
that the group be generated by two-generator subgroups of a particular kind. 
In view of the applications, 
I formulate  the result for a finitely generated group $\Gamma$ 
that need not be metabelian.

\cite{DiLe99}
\begin{thm}
\label{thm:Groups-with-infinite-cyclic-abelianisation}
Let $\Gamma$ be a finitely generated group 
which admits an epimorphism $\pi \colon \Gamma \epi \Z$,
a finite set $S$ of generators
and a subset $E_0$ of the set of all pairs $\{s,s'\}$ with elements in $S$,
having the following properties:
\begin{enumerate}[a)]
\item the epimorphism $\pi$ maps every generator $s \in S$ to $1 \in \Z$, and
\item the combinatorial graph $\Delta_0$ with vertex set $S$ 
and edge set  $E_0$ is connected.
\end{enumerate}
For every edge $e  = \{s, s'\} \in E_0$,
let $G_e$ be the subgroup generated  by the pair $\{s, s'\}$
and set $N_e = \ker  \pi \cap G_e$.
Then  the union $\bigcup_{e \in E_0} N_e$ generates the kernel of  $\pi$. 
\end{thm}

\begin{proof}
Let $M$ denote the subgroup of the kernel $N$ of $\pi$
which is generated by the subgroups $N_e$  with $e \in E_0$. 
I claim that $M$ is normal in $\Gamma$
and that it coincides with $N$.
Fix a generator $s_0 \in S$, an edge $e \in E_0$ with endpoints $s$ and $s'$,
and choose an element  $g \in N_e$. 
Then 
\[
s_0 \cdot g \cdot s_0^{-1} 
= 
(s_0 \cdot s^{-1}) \cdot (s \cdot g \cdot s^{-1}) \cdot (s_0 \cdot s^{-1})^{-1}.
\]
The factor $s \cdot g \cdot s^{-1}$ of this product  belongs to $N_e$.
Moreover,
as the elements $s_0$ and $s$ are vertices of the connected graph $\Delta_0$, 
there exists a path $(s_0, s_1, \ldots, s_f = s)$  in $\Delta_0$ 
which leads from $s_0$ to the vertex $s$ of the edge $e$.
This path allows one to express the first factor $s_0 \cdot s^{-1}$ 
as a product the the form
 \[
s_0 \cdot s^{-1} = (s_0 \cdot s_1^{-1}) \cdots (s_{f-1} \cdot s_f^{-1}).
\]
Recall now that $\pi$ maps every generator $s \in S$ to $1 \in \Z$.
Each factor  $s_i \cdot s^{-1}_{i+1}$ of the above product lies 
therefore in one of the kernels $N_e$ and hence in $M$,
and an analogous statement holds for $s_0^{-1}$.
It follows
that $M$ is a \emph{normal} subgroup of $\Gamma$,
and that $\Gamma/M$ is an cyclic group, 
generated by the common image of the generators $s_i $ in $\Gamma/M$.
On the other hand, the definition of $M$ implies that $M$ is contained in $N$ 
which, by definition, 
is the kernel of the epimorphism $\pi $ of $\Gamma$ onto $\Z$.
So the normal subgroup $M$ must coincide with $N$
and the proof is complete. 
\end{proof}

Theorem
\ref{thm:Groups-with-infinite-cyclic-abelianisation}
will be applied in section \ref{ssec:Artin-groups-many-odd-labels}
to Artin systems with infinite cyclic abelianisation
and in section \ref{ssec:Generalised-Artin-systems}
to a generalization of these Artin systems.
%
\subsection{Third result}
\label{ssec:Third-basic-result}
%
The third result lists conditions on a finitely generated group 
which imply 
that its metabelian top 
does \emph{not} admit a finite presentation.

To state the result,
I need a definition.
\begin{definition}
\label{def:Large-group}
Let $\Gamma$ be a finitely generated group. 
If $\Gamma$ maps onto a non-abelian free group
it is called \emph{very large};  
 if is contains a very large subgroup of finite index 
 it is termed \emph{large}.
\end{definition}

A finitely generated large group $\Gamma$ may not be very large.
This happens, for instance, 
whenever the metabelian top of $\Gamma$ is polycyclic 
and $\Gamma$ contains a non-abelian free subgroup of finite index.
Explicit examples of such groups are described in \cite[Example 4.10]{PaSu17}.
Actually, there are many other examples with the stated property; 
see, \eg Theorem IV.3.7 in \cite{Bau93}).
\smallskip

Now to the announced third result:
\begin{thm}
\label{thm:Infinitely-related-metabelian-top}
Suppose $\Gamma$ is a finitely generated group 
satisfying one of the conditions (i) through (v):
\begin{enumerate}[(i)]
\item $\Gamma$ is very large;
\item $\Gamma$ maps onto the wreath product of two infinite cyclic groups;
\item $\Gamma$ maps onto the wreath product $\Z_m \wr C_\infty$ 
of a finite cyclic group of order $m > 1$  and an infinite cyclic group; 
\item $\Gamma$ maps onto the semi-direct product $A \rtimes C_\infty$ of an
 \emph{infinite},  locally finite, abelian group $A$ by an infinite cyclic group;
 \item $\Gamma$ maps onto an infinitely related metabelian group.
\end{enumerate}
Then the metabelian top of $\Gamma$ does not admit a finite presentation. 
\end{thm}

\begin{proof}
I first prove the chain of implications 
\[
(i) \Rightarrow (ii) \Rightarrow (iii) \Rightarrow (iv) \Rightarrow (v).
\]

$(i) \Rightarrow (ii)$: if $\Gamma$ maps onto a non-abelian free group,
it maps onto a free group of rank 2 and hence onto the wreath product 
$\Z \wr C_\infty$ of two infinite cyclic groups.

$(ii) \Rightarrow (iii)$: for every $m > 1$,  
the group $\Z_m \wr C_\infty$ is a quotient  of $\Z \wr C_\infty$.

(iii) $\Rightarrow (iv)$: 
the wreath product $\Z_m \wr C_\infty$ is an extension of an infinite, locally finite abelian group by an infinite cyclic group.

(iv) $\Rightarrow (v)$: 
Let $Q$ be a finitely generated group 
that is an extension of an infinite, locally finite and abelian group $A$ 
by an infinite cyclic group generated by $t$.
\emph{Then $Q$ cannot be an ascending HNN-extension with a finitely generated subgroup $B \subset A$ and stable letter $t$}. 
Indeed,
 if $B$ is a finitely generated subgroup of the locally finite group $A$ 
 then it is finite;
 if, in addition, $B \subseteq t B t^{-1} $
the union of the conjugated groups $t^mB t^{-m}$ with $m \in \N$
will equal $B$ and so be finite,
whence this union cannot coincide with the infinite group $A$.  
If follows that $Q$ is not an ascending HNN-extension 
with finitely generated base group $B$ and stable letter $t$.
One sees similarly, that $Q$ is not an ascending HNN-extension with finitely generated base group and stable letter $t^{-1}$.
Theorem A in \cite{BiSt78}  allows one therefore to infer
that the metabelian group $Q$ does not admit a finite presentation.
\smallskip

Consider, finally, the metabelian top of $G =\Gamma/\Gamma''$ of $\Gamma$
and let $Q$ be a metabelian quotient of $\Gamma$ that does not admit a finite presentation.
The group $Q$ is then a quotient of $G$, 
the maximal metabelian quotient of $\Gamma$, 
and so  Lemma
\ref{lem:Infinitely-related-metabelian-top} allows one to infer
that $G$ is infinitely related, too.
\end{proof}
\begin{remarks}
\label{remarks:History-of-preceding-theorem}
The preceding theorem summarizes several earlier results.
Condition (i) is the hypothesis used by Papadima and Sugiu 
in Proposition 4.9 of \cite{PaSu17}.  
Condition (iii) figures in a paper by G. Baumslag 
(see \cite[Thm.\;A]{Bau72});
actually Baumslag proves more: if condition (iii) is fulfilled, 
the homology group $H_2( - ,\Z)$ of the metabelian top is not finitely generated.
Condition (iv) is used in \cite[Thm.\;D]{BaMi09}
and also in \cite[Corollary B3.32]{Str13a}.
Condition (v), finally, has been added to the list in order to draw attention to the following fact:
in conditions (iii) and (iv) the groups $G$ are extensions of abelian torsion-groups by an infinite cyclic group.
Now there are infinitely related metabelian groups 
which are neither of this kind 
nor free metabelian nor the wreath product of two infinite cyclic groups.
Prototypes of such groups are the metabelian tops of the Baumslag-Solitar groups
\begin{equation}
\label{eq:Baumslag-Solitar-group}\Gamma =\langle a, t \mid t a^{m} t^{-1} = a^{n}\rangle
\end{equation} 
with $m > 1$, $n > 1$ and $\gcd(m,n) = 1$.
Then the metabelian top is an extension of the form $\Z[1/(m \cdot n)] \rtimes C_\infty$ and it is infinitely related
(see \cite{BaSo62} or \cite{BaSt76}).
\end{remarks}
\section{Finitely presented groups all whose free subgroups are cyclic}
\label{sec:Fp-groups-with-no-non-abelian-free-subgroups}
The class of groups considered in this section
comprises finitely presented \emph{soluble} groups,
but also many other groups.
The starting point of our discussion is Theorem B in  \cite{BiSt80},
a result that can be stated like this:
\begin{thm}
\label{thm:ThmB-BiSt80}
Let $\Gamma$ be a finitely presented group, 
or a group of type $\FP_2$, which does not contain a non-abelian free subgroup.
Then every metabelian quotient of $\Gamma$ admits a finite presentation.
\end{thm}
Here are two consequences of Theorem 
\ref{thm:ThmB-BiSt80}.

\subsection{Finitely presented soluble groups}
\label{ssec:Fp-soluble-group}
Well-known examples of finitely presented soluble groups 
are poly-cyclic groups, 
but there exists far bigger specimens, 
in particular groups with infinite torsion-free rank
\footnote{see, \eg \cite[p.\;422, Ex.\;1]{Rob96}
 for a definition of this notion}.
 The first such groups have been detected in the 70s by G. Baumslag 
 and, independently, by V. M. Remeslennikov (see \cite{Bau72}, \cite{Bau73}, \cite{Rem73a}); the groups studied by these authors were metabelian.
 Later on,
 non-metabelian, soluble matrix groups with infinite torsion-free rank 
 came to light (see, in particular, Theorem 1 in \cite{BGS86}).

\subsection{Groups of PL-homeomorphisms of the real line}
\label{ssec:PL-homeomorphism-groups}
Being soluble is a property which rules out
that the group in question does contain a non-abelian-free subgroup;
being a member of a variety of groups 
distinct from the variety of all groups 
has the same effect.
M. Brin and Squier establish in\cite{BrSq85} 
that groups of a quite different sort also enjoy this consequence,
namely  
the group $\PL_o(\R)$ of all increasing PL-homeomorphisms of the real line
\footnote{the number of break-points of each individual PL-homeomorphism is required to be finite}
and all its subgroups (see \cite[Thm.\;3.1]{BrSq85}).
The group $\PL_o(\R)$ is, of course, infinitely generated, 
but it contains finitely generated subgroups with finite presentations,
in particular R. Thompson's group $F$.
 \footnote{See \cite[§ 3]{CFP96} and the results
\cite[Thm.\;2.9]{BrSq85}, \cite[Prop.\;4.8]{Bro87a},
\cite[Thm:\;2.5]{Ste92},
as well as Proposition D13.7, Theorem D14.2 and  Proposition  D15.10 
in \cite{BiSt16}.}
%

\section{One relator groups}
\label{sec:One-relator-groups}
%
In this section, I discuss two classes of one relator groups
whose metabelian tops are finitely presented for some members of the class,
 but infinitely related for other members, 
 and for which there exists an algorithm 
that allows one,
given a particular member of the class,
 to decide whether the first or the second alternative holds. 

Let $\Gamma$ be a group with $m$ generators $x_1$, \ldots, $x_m$
and one defining relator $r$.
If $m = 1$, the group is cyclic; 
if $m > 2$, the metabelian top is infinitely related,
as will be shown in Section \ref{ssec:Deficiency-bigger-than-1}.
There remains the case where $m = 2$.
Then the metabelian top can be finitely related
and whether this is the case can be read of from the relator $r$.
Actually, 
two cases arise, 
depending on whether or not the relator is a product of commutators.
 %
\subsection{Relator is a product of commutators}
\label{ssec:One-relator-groups-relator-in-[F,F]}
Let $\Gamma$ be a group with generators $x$ and $y$, 
and a single defining relator $r$
that has exponent sum 0 with respect to $x$ and with respect to $y$.
Then $\Gamma_{\ab}$ is free abelian of rank 2 
and $(\Gamma')_{\ab}$ is a cyclic $\Z{G_{\ab}}$--module 
generated by the image of the commutator $[x,y]$ in $(G')_{\ab}$.
According to \cite[Lemma 1]{Str81a},
the annihilator ideal  of $\Gamma'_{\ab}$ is principal, 
and it is generated by an element $\lambda \in \Z{G_{\ab}}$ 
that can be obtained from the relator $r$ as follows.

Let $F$ be the free group on $\{x, y\}$,
let $D_x \colon  F \to \Z{F}$ denote the partial derivative with respect to $x$,
and let $\widehat{\phantom{-}} \colon \Z{F} \epi \Z{F_{\ab}}$ denote the canonical projection.
The element $\widehat{D_x(r)}$ 
is then a multiple of $1 - \hat{y}$
and 
\begin{equation}
\label{eq:Defining-annihilator}
\lambda = \widehat{D_x(r)}/(1- \hat{y})
\end{equation}
generates the annihilator  of the module $(\Gamma')_{\ab}$.
The following isomorphism of $\Gamma'_{\ab}$ thus holds:
\begin{equation}
\label{eq:Abelianized-commutator-subgroup}
A  = \Z{\Gamma_{\ab}}/(\Z{\Gamma}_{ab} \cdot \lambda) \iso \Gamma'_{ab},
\quad 1 +  (\Z{\Gamma}_{ab} \cdot \lambda) \longmapsto [x,y] \cdot \Gamma'' .
\end{equation}

Thanks to the theory developed in \cite{BiSt80}  and \cite[§ 5]{BiSt81a},
the knowledge of $\lambda$ allows one to determine 
whether or not $\Gamma/\Gamma''$ admits a finite presentation.
To describe the decision procedure, 
I need some additional notation.
Let $\vartheta \colon F_{\ab} \iso \Z^2$ be
the isomorphism of the free abelian group $F_{\ab}$ 
onto the standard lattice $\Z^2$  in the euclidean plane $\R^2$
which sends $x\cdot [F,F]$ to $(1,0)$ and $y\cdot [F,F]$ to $(0,1)$.
The element $\lambda$ is a linear combination of elements in $ Q = F_{\ab}$,
say $\lambda = \sum_q \lambda_q  \cdot q$. 
Set
\[
\supp(\lambda) = \{\vartheta(q) \mid \lambda_q \neq 0\}.
\] 
According to Theorem B in \cite{Str81a} 
the following characterization holds:
\begin{prp}
\label{prp:Characterisation-of-Gamma/Gamma''-fr-case-rank-2}
Assume $\supp(\lambda)$ is not a singleton.
Then its metabelian top $\Gamma/\Gamma''$ is finitely related if, 
and only if, $\supp(\lambda)$ has the following properties:
\begin{enumerate} [(i)]
\item $\supp(\lambda)$ does not lie on a straight line,
\item the convex polygon $\PP$ bounding the convex hull of $\supp(\lambda)$ has no parallel edges, and
\item for every edge $e$ of $\PP$ 
the vertex $v_e$ of $\PP$ with greatest distance from the straight line supporting the edge $e$ has coefficient $1$ or $-1$.
\end{enumerate}
\end{prp}
\begin{illustration}
\label{illustration:Baumslag-Boler-group}
On p.\;68 of his survey \cite{Bau74}, 
G. Baumslag discusses the group $\Gamma$ with generators $x$, $y$ 
and defining relator
\begin{equation}
\label{eq:Baumslag-Boler-relator}
r = \act{x} [y,x] \cdot \act{y}[x,y] \cdot [x,y] 
\end{equation}
The relator $r$ of this group is a product of commutators 
and so $\Gamma_{\ab}$ is free abelian of rank 2.
The partial derivative $D_x(r)$ of the relator $r$ is easily found:
\begin{align*}
D_x(r) 
&= 
D_x \left( 
(xyx y^{-1} x^{-2}) \cdot (y xy x^{-1} y^{-2}) \cdot (xyx^{-1} y^{-1}\right)\\
&=
\left(1 + xy - x[y,x] - \act{x}[y,x]\right) 
+
\act{x}[y,x] \cdot y \left( 1-xyx^{-1}\right) + \\
& \phantom{==}
\act{x} [y,x] \cdot \act{y}[x,y] \cdot ( 1-xyx^{-1}).
\end{align*}
The image of $D_x(r)$ in $\Z{F_{\ab}}$ simplifies to
\[
\widehat{D_x(r)} 
=
(1 + \hat{x} \hat{y}- \hat{x} - 1) + \hat{y}(1-\hat{y}) + (1- \hat{y}) 
= (1-\hat{y}) \cdot (1 - \hat{x} +\hat{y} )
\]
and so $\lambda = 1 - \hat{x} + \hat{y} $.
The support of $\lambda$ is thus  a triangle 
the vertices of which have coefficients $1$ or $-1$.
According to 
Proposition \ref{prp:Characterisation-of-Gamma/Gamma''-fr-case-rank-2}
the metabelian top of $\Gamma$ admits therefore a finite presentation.
\begin{remark}
\label{remark:Baumslag-Boler-group}
The abelianized derived group $\Gamma'_{\ab}$ of the preceding example 
is isomorphic to the abelian group underlying the module 
$\Z{\Gamma_{\ab}} / \Z{\Gamma_{\ab}} (1 - \hat{x} + \hat{y})$;
this group is free abelian of countable rank.
\end{remark}
\end{illustration}

\subsection{Relator is not a product of commutators}
\label{ssec:One-relator-groups-relator-not-in-[F,F]}
Let $\Gamma$ be a group with generators $x$ and $y$, 
and a single defining relator $r$ which is not a product of commutators.
The analysis of the metabelian top of such a group is,
in general, 
more complicated than that of the case treated before 
in view of the fact 
that the abelianized group may not be infinite cyclic,
but isomorphic to a group of the form $\Z \oplus \Z_d$ with $d > 1$.
The reader can find an investigation of this general case
in Sections 5 and 6 of \cite{BiSt78}.

In the special case where $\Gamma_{\ab}$ is infinite cyclic
the analysis simplifies considerably and becomes pleasant,
as I now show.
Let $F$ be the free group with basis $\{ x, y\}$,
let $w$ be a cyclically reduced word in $x$ and $y$,
and define $\Gamma$ to be the one relator group with the presentation
$\langle x, y \mid w \rangle$.
If $w$ consists of a single letter, the group $\Gamma$ is infinite cyclic,
a case that needs no further study.
Assume therefore that the length of $w$ is greater than 1
and consider the exponent sums $\sigma_x(w)$ and $\sigma_y(w)$ 
of $w$ with respect to $x$ and to $y$.
\footnote{By definition, $\sigma_x(w)$ is the difference of the number of letters $x$ minus the numbers of letters $x^{-1}$ in the word $w$.
The definition of $\sigma_y(w)$ is analogous.}
The abelianisation of $\Gamma$ is then isomorphic to
\[
\Z^2/\Z (\sigma_x(w),\sigma_y(w))
\]
and so it is infinite cyclic if, and only if, 
$\sigma_x(w)$ and $\sigma_y(w)$ are relatively prime.

Assume $\Gamma_{\ab}$ is infinite cyclic. 
According to Theorem 3.5 in Section 3.3 of \cite{MKS04} 
one can then find a set of free generator $\{a, t\}$ of $F$ 
such that $w$, 
when expressed as a word in $a$ and $t$, 
has exponent sum 1 with respect to $a$  
and exponent sum 0 with respect to $t$.

Consider now the well-known beginning of a $\Z{\Gamma}$-free resolution of the trivial $\Gamma$-module $\Z$ which goes back to Lyndon's paper \cite{Lyn50}. 
Using the generators $a$ and $t$ introduced before,
the end of the resolution has this form
\begin{equation}
\label{eq:Lyndons-resolution}
\xymatrixcolsep{4pc}%
\xymatrix{
\Z{\Gamma}  
\ar@{->}[r]^-{(D_a(r),D_t(r))}
&
\Z{\Gamma} \oplus \Z{\Gamma} 
\ar@{->}[r]^-{\left(\begin{smallmatrix}1-a\\1-t \end{smallmatrix}\right)} 
&
\Z{\Gamma} \ar@{->}[r]^-{\varepsilon} 
& 
\Z  \longrightarrow 0 
}
\end{equation}
We use this partial resolution to compute $\Gamma'_{\ab}$,
a group that is isomorphic to $H_1(\Gamma', \Z)$ or, by Shapiro's lemma,
isomorphic to $H_1(\Gamma, \Z{\Gamma_{\ab}})$.
Let 
$\hat{\phantom{x}} \colon \Gamma \epi \Gamma_{\ab}$ 
denote the canonical projection.
The group $\Gamma_{\ab}$ is infinite cyclic, 
generated by the image $\hat{t}$ of $t$,
and $\hat{a} = 1$.
By tensoring the exact complex \eqref{eq:Lyndons-resolution} with 
$\Z{\Gamma_{\ab}} \otimes_{\Z{\Gamma}} -$ one obtains, in particular, 
the complex
\begin{equation}
\label{eq:complex-deduced-from-resolution}
\xymatrixcolsep{4pc}%
\xymatrix{
\Z{\Gamma_{\ab}}  
\ar@{->}[r]^-{(\widehat{D_a(r)},\widehat{D_t(r)})}
&
\Z{\Gamma_{\ab}} \oplus \Z{\Gamma_{\ab}} 
\ar@{->}[r]^-{\left(\begin{smallmatrix}0\\1-\hat{t} \end{smallmatrix}\right)} 
&
\Z{\Gamma_{\ab}}
}
\end{equation}
It follows that the $\Z{\Gamma_{\ab}}$-module $\Gamma'/\Gamma''$
is isomorphic to
\begin{equation}
\label{eq:Structure-of-Gamma'-abelianized} 
A = \Z{\Gamma_{\ab}} / \left(\Z{\Gamma_{\ab}} \cdot \widehat{D_a(r)} \right).
\end{equation}
To simplify notion, 
I denote henceforth the infinite cyclic group $\Gamma_{\ab}$ by $C$,
write $c$ for its generator $\hat{t}$ 
and denote the group ring element $\widehat{D_a(r)}$  by $f(c)$.
Note that $f(c)$ is a Laurent polynomial with coefficients in $\Z$.

The metabelian top $\Gamma/\Gamma''$ of $\Gamma$
is the split extension of the $\Z{C}$-module $A = \Z{C}/(\Z{C} \cdot f(c))$ 
by the infinite cyclic group $C$. 
There remains the problem of finding out 
when this extension is finitely related.
Here is the answer: 
\begin{prp}
\label{prp:Characterization-of-fp-metabelian-groups}
Let $C$ be an infinite cyclic group generated by the element $c$
and let $A$ be the cyclic module $\Z{C}/ (\Z{C} \cdot f(c))$.
Then the following conditions are equivalent:
\begin{enumerate}[(i)]
\item The group $A \rtimes C$ admits a finite presentation;
\item the leading or the trailing coefficient of $f(c)$ is either 1 or $-1$.
\end{enumerate}
\end{prp}
\begin{proof}
By \cite[Theorem C]{BiSt78}
the group  $ G= A \rtimes C$  admits a finite presentation
if, and only if, $G$ is an ascending HNN-extension with finitely generated base group $B \subseteq A$ and stable letter either $c$ or $c^{-1}$.

The element $f(c) \in \Z{C}$ is a Laurent polynomial.
Suppose $f(c)$ is not the zero element.
Then it has a representation of the form
\[
f(c) = a_\mu c^\mu + a_{\mu + 1} c ^{\mu + 1} + \cdots + a_M c^M
\quad \text{with}\quad
a_\mu \neq 0 \text { and } a_M \neq 0.
\]
Assume first that the leading coefficient $a_M$ of $f(c)$ is a unit in $\Z$
and let $\tilde{B}$ be the subgroup of $\Z{C}$ generated by the monomials
$c^j$ with $\mu \leq j < M$.
The element  $c^{M} - c_M \cdot f(c)$ lies then in $\tilde{B}$.
More generally,
every monomial $c^k$ with $k \geq M$ can be written in the form  
$c^k = g(c) \cdot f(c) + \tilde{b} $ for some $g(c) \in \Z{C}$ and 
$\tilde{b} \in \tilde{B}$.
Let $B$ denote the canonical image of $\tilde{B}$ 
in the module $A = \Z{C}/ (\Z{C} \cdot f(c))$.
The preceding calculations then show
that $B$ is a finitely generated subgroup 
that contains every coset $c^k + \Z{C} \cdot f(c)$ with $k \geq 0$,
and so $B$ is a finitely generated subgroup of $A$ with $c \cdot B \subset B$. 
It follows that 
\begin{equation}
\label{eq:Ascending-union-is-A}
B \subseteq c^{-1}  B \subset c^{-2} B \subseteq \cdots
\quad \text{ and } \quad
\bigcup\nolimits_{k \in \N} c^{-k}  B = A
\end{equation}
and so $A \rtimes C$ is an ascending HNN-extension with a finitely generated base group $B$ and stable letter $c^{-1}$.

Suppose now that the leading coefficient $a_M$ of $f(c)$ is not in $\{-1,0, 1\}$.
Consider a finitely generated subgroup $\tilde{B} \neq{0}$  of $\Z{C}$.
There exists then an upper bound, say $d$,  
for the exponents of the monomial $c^k$ 
that occur in the non-zero elements of $\tilde{B}$. 
Consider a linear combination of the form
$h(c) = c^{\ell} + g(c) \cdot f(c)$ with $\ell > d$.
The leading coefficient of $g(c) \cdot f(c)$ cannot be $-1$ 
and so $h(c)$ must involve a term $a_ k c^k$ 
with $a_k \neq 0$ and $k \geq \ell$.
Pass to the canonical image $B$ of $\tilde{B}$ in $A$.
Then $B$ does not contain any of the elements $c^{\ell} + \Z{C} \cdot f(c)$ 
with $\ell > d$.
This, however, implies that $B$ cannot be the base group of an ascending HNN-extension with stable letter $c^{-1}$.
Indeed, suppose this were possible.
Then one would have an ascending chain of  subgroups 
$ B \subseteq c^{-1} B \subseteq \cdots$  whose union is $A$. 
There would therefore exist a natural number $k_0$ 
such that $c^{d+1 } + \Z{C} \cdot f(c)$ lies in $c_{-k_0} \cdot B$
and so $c^{d+1+k_0} + \Z{C} \cdot f(c)$ would be an element of $B$,
contrary to what has been established before.

Taken together, the last two paragraphs show the following:
if $f(c)$ is not the zero polynomial 
then $A \rtimes C$ is an ascending HNN-extension 
with finitely generated base group and stable letter $c^{-1}$
if, and only if, the \emph{leading} coefficient of $f(c)$ is 1 or $-1$.
Similarly one sees that, provided $f(c) \neq 0$,
the group $A \rtimes C$ is an ascending HNN-extension 
with finitely generated base group and stable letter $c$
precisely if the trailing coefficient of $f(c)$ is 1 or $-1$.

One is left with the case where $f(c) = 0$.
Then $A$ is the group ring $\Z{C}$ and it is evident 
that no finitely generated non-zero subgroup $B$ of $A$ satisfies 
$B \subseteq cB$ or $B \subseteq c^{-1} B$.
\end{proof}
\begin{illustration}
\label{illustration:Characterization-of-fp-metabelian-groups}
Let $\Gamma$ be the one-relator group with generating set $\{a, t\}$ 
and defining relator
\begin{equation}
\label{eq:Defining-relator-one-relaor-group}%
r = t a^2 t^2 a t^{-3} a^{-2}.
\end{equation}
The exponents sums of this word are 1 with respect to $a$ and 0 with respect to $t$;
formulae \eqref{eq:Lyndons-resolution}
 through \eqref{eq:Structure-of-Gamma'-abelianized}  apply therefore.
 One finds that
 \[
 D_a(r) = t(1 + a) + ta^2 t^2 - t a^2 t^2 a t^{-3} (a^{-1} + a^{-2})
 \quad \text{and} \quad
 \widehat{D_a(r)} = 2 t + t^3 - 2;
 \]
Proposition \ref{prp:Characterization-of-fp-metabelian-groups}
 thus allows one to conclude 
 that the metabelian top $\Gamma/\Gamma''$ of $\Gamma$ is finitely related.
\end{illustration}

\section{Knot groups}
\label{sec:Knot-groups}
%
Suppose $\Gamma$ is a knot group and $f$ is its Alexander polynomial
(see, \eg \cite{BZ03} for unexplained terminology).
If the knot giving rise to $\Gamma$ can be chosen to lie in a plane 
and hence is unknotted,
the group $\Gamma$ is infinite cyclic;
otherwise the degree of the Alexander polynomial $f$ is positive 
and $\Gamma$ contains non-abelian free subgroups
(this follows, \eg from Theorem 1 in \cite{Neuw60}).

Now to the metabelian top of $\Gamma$.
There is a first pleasant result, due to H. F. Trotter,
which states  that
\emph{the group $\Gamma/\Gamma''$ is finitely related 
if, and only if, the abelian group $\Gamma'_{\ab}$ is finitely generated}
(see the last two lines of the introduction of \cite{Tro74}).
And an equally pleasing second result,
due to E. Strasser-Rapaport,
which asserts
that \emph{the the group $\Gamma'/\Gamma''$ is finitely generated precisely
when the leading coefficient of the Alexander polynomial is $1$ or $-1$}
(see \cite{Str60}).
Note that the abelian group $\Gamma'_{\ab}$ is always torsion-free 
(cf.\;\cite[Corollary 4.2]{Str74})
and of finite torsion-free rank (see \cite[Theorem 3]{Str60}).

The cited results lead to the following
\begin{prp}
\label{prp:Characterization-of-polycyclic-metabelian-top-for-knot-groups}
Assume $\Gamma$ is a knot group.
Then $\Gamma/\Gamma''$ is finitely related precisely it is is polycyclic, 
and this happens if, and only if, the leading coefficient of the Alexander polynomial is $1$ or $-1$.
\end{prp}

\section{Artin systems}
\label{sec:Artin-systems}
%
%
I begin by recalling the notion of an \emph{Artin system}.
Let $\Delta(S, E, \lambda)$ be a finite, labelled combinatorial graph,
with non-empty set of vertices $S$, 
set of edges  
\[
E \subset \{\{s,s'\}
\mid s \in S, s' \in S \text{ and } s \neq s'  \},
\]
and labelling function $\lambda \colon E \to (\N \setminus \{0,1\})$.
Enumerate the elements of $S$.
Use this enumeration to define, for every pair in $E$, 
a relation $R_e$ of a particular kind:
given elements $s_i$ and $s_j$ with $i < j$,
let $u(s_i, s_j)$ be the positive, alternating word in $s_i$ and $s_j$ 
that starts with $s_i$ and has length $\lambda (e)$, 
so $u(s_i, s_j)= s_i s_j s_i \cdots$,
and let $u(s_j, s_i)$ be the word that arises from  $u(s_i, s_j)$ by exchanging the letters $s_i$ and $s_j$; thus $u(s_j, s_i) = s_j s_i s_j \cdots$.
Define $R_e$ to be the relation
\begin{equation}
\label{eq:Relation-R_e}
u(s_i,s_j)  = u(s_j, s_i)
\end{equation}
and let $\Gamma$ be the quotient of the free group on $S$ 
by the congruence relation generated by the relations $R_e$.

\begin{definition}
\label{definition_Artin-system}
The Artin system $(\Gamma, \Delta(S, E, \lambda))$ consists 
of the group $\Gamma$, called \emph{Artin group}, and the graph $\Delta$
with vertex set $S$, edge set $E$ and labelling function $\lambda$,
and, for each $e  = \{s_i, s_j\} \in E$ with $i < j$,
the defining relation $u(s_i,s_j) = u(s_j,s_i)$.
\end{definition}
 
In this section, various classes of Artin systems will be studied.
The proofs of some results will be based on a consequence of 
Theorem \ref{thm:FR-implies-polycyclic},
namely Theorem \ref{thm:Metabelian-top-of-Artin-group} below.
The study of Artin systems will be continued in 
Section \ref{sec:Artin-systems-of-finite-type}
with a closer look at a special class of Artin systems.
%
\subsection{Inversion $\iota$ of an Artin system}
\label{ssec:Inversion-admitted-by-Artin-systems}
I begin with a basic result about Artin systems.
Let  $(\Gamma, \Delta(S, E, \lambda))$ be such a system.
Then the assignments $s \mapsto s^{-1}$  for $s \in S$
extend to an automorphism $\iota$ of $\Gamma$.

Indeed, let $F = F(S)$ be the free group on the finite set $S$
ad let $\tilde{\iota} \colon F \iso F$ be the automorphism 
which extends the assignments $s \mapsto s^{-1}$ for $s \in S$.
I claim that the automorphism $\tilde{\iota} \colon F \iso F$ 
maps every defining relation of $\Gamma$ to a consequence of the very same defining relation.
The defining relations of $\Gamma$ are of two forms, 
depending on whether the label of $\{s, s'\}$ is even or odd.
Consider first a relation where the label of $\{s,s'\}$ is even.
The calculations
\begin{align*}
\tilde{\iota}((s s')^m) &=  (s^{-1} (s')^{-1})^m = (s' s)^{-m},\\
\tilde{\iota}((s' s)^m) &=  ((s')^{-1} s^{-1} )^m = (s s')^{-m}
\end{align*}
show
that the relation $\iota((s s')^m) = \iota((s' s)^m) $ 
is a consequence of the defining relation $(s s')^m = (s' s)^m$. 
Consider next a relation with odd label.
Then
\begin{align*}
\tilde{\iota}((s s')^m s) &=  (s^{-1} (s')^{-1})^m s^{-1} = \left((s s')^{m} s\right)^{-1}, \\
\tilde{\iota}((s' s)^m s') 
&=  
((s')^{-1} s^{-1} )^m (s')^{-1} =\left((s' s)^{m} s'\right)^{-1}
\end{align*}
and so the relation $\iota((s s')^m s)  = \iota((s' s)^m s')$ is a consequence of 
the defining relation $(s s')^m s =  (s's)^m s'$.

It follows that the group $\Gamma$  admits an automorphism $\iota$ 
which sends each element $s \in S$ to its inverse.
This automorphism induces, of course, 
an automorphism $\iota_2 \colon G \iso G$ 
of the metabelian top $G$  of $\Gamma$
and an automorphism $\iota_1 \colon \Gamma_{\ab} \iso \Gamma_{\ab}$ of the abelianisation of $\Gamma$. 
This latter automorphism sends 
each element of $\Gamma_{\ab}$ to its inverse;
in additive notation, it is thus $-\id$.
So Theorem \ref{thm:FR-implies-polycyclic} applies 
and leads to the next result:
 \begin{thm}
\label{thm:Metabelian-top-of-Artin-group}
Let $(\Gamma, \Delta(S, E, \lambda))$ be an Artin system
and let $G = \Gamma/\Gamma''$ be the met\-abelian top of the group $\Gamma$.
Then $G$ is finitely related if, and only if, it is polycyclic.
\end{thm}

\begin{remark}
\label{remark:Metabelian-top-of-Artin-group}
The conclusion of Theorem \ref{thm:Metabelian-top-of-Artin-group}
is analogous to that of Trotter's result, 
mentioned in Section \ref{sec:Knot-groups}.
\end{remark} 

In view of Theorem \ref{thm:Metabelian-top-of-Artin-group}
the problem of finding out 
whether the metabelian top $G$ of an Artin group is finitely related 
reduces to the question whether $G$ is polycyclic or, equivalently,
whether the abelian group $\Gamma'_{\ab}$ is finitely generated.
In sections \ref{ssec:Two-generator-Artin-systems}, 
\ref{ssec:Artin-groups-many-odd-labels}
and in Section \ref{sec:Artin-systems-of-finite-type},
various classes of Artin systems will be discussed
for which I have been able to determine
whether or not their metabelian tops are polycyclic.
In section \ref{ssec:Generalised-Artin-systems}, finally, 
I study a class of generalized Artin groups 
all whose metabelian tops are polycyclic.

\begin{remark}
\label{remark:Subgroups-of-Artin-groups}
An Artin system $(\Gamma, \Delta(S, E, \lambda))$ is,
by definition, 
given by a presentation with generating set $S$ 
and a set of relations $\RR$ of a very particular kind,
one particularity being
that every relation involves exactly two generators in $S$.
If $S_1$ is  a non-empty subset of $S$ 
there is thus a well-defined subset $\RR_1$ of $\RR$ made up of all the relations that involve only generators in $S_1$.
The sets $S_1$ and $\RR_1$ define then an Artin system $(\Gamma_1, S_1)$.
In addition, 
there is a homomorphism $\can_{S_1 \incl S} \colon\Gamma_1 \to \Gamma$
induced by the inclusion of $S_1$ into $S$.
\emph{This homomorphism is injective}; see \cite[Theorem 3.1]{Par97b}.
In the sequel, 
I shall use this important result to simplify the wording in several statements or proofs;
one could avoid making use of the result 
at the expense of inserting the word \emph{canonical} at appropriate places.
\end{remark}

%
\subsection{Artin systems with two generators}
\label{ssec:Two-generator-Artin-systems}
If the Artin system $(\Gamma, \Delta(S, E, \lambda))$ is generated by two elements,
 say by $s_1$ and by $s_2$,
 the group is a one-relator group. 
 So the results of Section \ref{sec:One-relator-groups} apply
 and allow one to determine 
 whether the metabelian top admits a finite presentation.
 In the present case, however, the relator is of a very simple form 
and so the answer simplifies.
Two cases arise, 
depending on whether the label of the edge $\{s_1, s_2\}$ is even or odd.

\emph{Assume first the label is even}, say $2m$ with $m > 0$.
The defining relation of $\Gamma$ has then the form 
$(s_1 s_2)^m  = (s_2s_1)^m$ 
and this relation is equivalent to the relator
\begin{equation}
\label{eq:Artin-relator-with-even-label}
r = (s_1 s_2)^m \cdot (s_2s_1)^{-m} 
= 
\left(s_1 s_2\right)^m \cdot \left(s_1^{-1} s_2^{-1} \right)^m .
\end{equation}
The reasoning at the beginning of section 
\ref{ssec:One-relator-groups-relator-in-[F,F]}
allows one to determine the structure of the abelian group 
underlying the $\Z{\Gamma_{\ab}}$-module  $\Gamma'_{\ab}$.
To find it, 
one calculates the image of the partial derivative $D_{s_1}(r)$ 
under the canonical projection 
$\widehat{\phantom{-}} \colon \Z{F} \epi \Z{F_{\ab}}$.
One finds first of all:
\begin{equation*}
D_{s_1}(r)
= 
\left(1 + s_1 s_2 + \cdots + (s_1 s_2)^{m-1} \right)
-
(s_1 s_2)^m \cdot \left(s_1^{-1} + \cdots  + (s_1^{-1} s_2^{-1})^{m-1} s_1^{-1}\right)
\end{equation*}
Set $t_1 = \hat{s}_1$ and $t_2 = \widehat{s}_2$.
The image of $D_{s_1}(r)$ under the canonical projection is then:
\begin{align*}
\widehat{D_{s_1}(r)}
&=
\left(1 + t_1t_2 + \cdots + (t_1 t_2)^{m-1} \right) 
-
\left((t_1 t_2)^{m-1} t_2+ \cdots + (t_1 t_2) t_2 + t_2\right)\\
&= 
\left(1 + t_1t_2+ \cdots + (t_1 t_2)^{m-1}\right) \cdot  (1-t_2).
\end{align*}

By formula \eqref{eq:Abelianized-commutator-subgroup}
the $\Z{\Gamma_{\ab}}$-module $\Gamma'_{\ab}$
is thus isomorphic to
\begin{equation}
\label{eq:Module-Artin-group-with-2-generators-Case1}
A= \Z{\Gamma_{\ab}}/ (\Z{\Gamma_{\ab}} \cdot \hat{\lambda}) 
\quad\text{with}\quad
\hat{\lambda} = 1 + t_1t_2 + \cdots + (t_1t_2)^{m-1}.
\end{equation}
Two cases now arise.
If $m = 1$, then $\hat{\lambda} = 1$ and $A$ is reduced to 0, 
in accordance with the fact 
that $G = \Gamma/\Gamma''$ is abelian.
If $m > 1$
the abelian group underlying the module $A$ 
is free abelian of countably infinite rank.
Indeed,
let $U$ be the subgroup of $\Gamma_{\ab}$ 
generated by the element of $t_1\cdot t_2$.
The additive group of the ring $\Z{U}/(\Z{U} \cdot \hat{\lambda})$ 
is free abelian of rank $m-1$.
On the other hand, 
$U$ is a direct summand of  $\Gamma_{\ab}$ 
and $\Gamma_{\ab}/U$ is infinite cyclic.
The abelian group underlying $A$ is thus free abelian of countable rank.
In view of Theorem \ref{thm:Metabelian-top-of-Artin-group}
this fact forces the metabelian top of $\Gamma $ to be infinitely related.

The previous reasoning establishes
 \begin{prp}
\label{prp:Two-generator-Artin-groups-with-even-label}
Let $(\Gamma, \Delta(S, E, \lambda))$ be the Artin system with standard generators $s_1$ and $s_2$,
label $\lambda(\{s_1,s_2\}) = 2m$ 
and defining relator \eqref{eq:Artin-relator-with-even-label}. 
If $m =1$, the group $\Gamma$ is abelian;
if, on the other hand, $m > 1$, 
the abelian group $\Gamma'_{\ab}$ is free abelian of countable rank
and the metabelian top of $\Gamma$ is infinitely related. 
\end{prp}
\smallskip

\emph{Assume now that the label is odd}, say $2m +1$ with $m \geq 1$.
The defining relation of $\Gamma$ is then
\begin{equation}
\label{eq:Artin-relation-with-odd-label}
(s_1 s_2)^m \cdot s_1 = s_2 \cdot  (s_1  s_2)^m.
\end{equation}
The letter $s_1$ occurs $m+1$ times on the left hand side of this relation
and $m$ times on its right side,
while the letter $s_2$ occurs $m$ times on the left
and $m+1$ times on the right.
It follows that $s_1$ and $s_2$ define the same element in $\Gamma_{\ab}$
and that $\Gamma_{\ab}$ is infinite cyclic.
The next result shows that
 the group $\Gamma'$ is finitely generated.
\begin{prp}
\label{prp:Two-generator-Artin-groups-with-odd-label}
Let $(\Gamma, \Delta(S, E, \lambda))$ be the Artin system with standard generators $s_1$ and $s_2$,
label $\lambda(\{s_1,s_2\}) = 2m + 1$ 
and defining relation \eqref{eq:Artin-relation-with-odd-label}. 
Then the derived group of $\Gamma$ is a free group of rank $2m$.
\end{prp}

\begin{proof}
To simplify notation, set $s = s_1$ and $s' = s_2$.
The abelianisation of $\Gamma$ is infinite cyclic, 
and the derived group of $\Gamma$ is generated 
by the conjugates of  $ a = s_2 \cdot s_1^{-1} = s' \cdot s^{-1} $ by the powers of $s$.
Notice that $s' = as$. 
Set  $v_1 = s(s's)^m $ and $v_2 = (s's)^m s'$.
Then
\begin{align*}
v_1 &= s \cdot  (s's)^m =  s\cdot (as^2)^m
= 
\act{s} a \cdot \act{s^3} a \cdots \act{s^{2m-1}}a \cdot s^{2m+1}\\
v_2 
&= 
(s's)^m\cdot  s' = (as^2)^m \cdot as 
= 
a \cdot  \act{s^2} a \cdots  \act{s^{2m}}  a \cdot s^{2m+1}
\end{align*}
and so 
\[
v_1 \cdot v_2^{-1} 
= 
\left(\act{s} a \act{s^3} u \cdots \act{s^{2m-1}}a \right)
\cdot 
\left(
\act{s^{2m}}a^{-1} 
\cdot 
\act{s^{2m-2}}a^{-1} 
\cdots 
\act{s^2}a^{-1} \cdot a^{-1}\right)
\]
Upon setting $a_n =  s^n a s^{-n}$ 
the relator $r_0 = v_1 \cdot v_2^{-1} $ can be written in the form
\begin{equation}
\label{eq:Relator-r0}
r_0 
=
\left(a_1 \cdot  a_3 \cdots a_{2m-1} \right)
\cdot 
\left(a_{2m}^{-1} \cdot a_{2m-2}^{-1} \cdots a_2^{-1} \cdot a_0^{-1}\right)
 \end{equation}
 The derived group of $\Gamma$,
 has therefore a presentation with  $(a_n)_{n \in \Z}$  as the set of generators
 and with the conjugates of $r_0$ by the powers of $s$ as the defining relators.
 It follows that
 the derived group $\Gamma'$ is a free group of rank $2m$, 
 freely generated by the canonical images of  $u_0$, $u_1$, \ldots, $u_{2m-1}$ in $\Gamma'$.
 \end{proof}
 %
 %
\subsection{Artin groups with many odd labels}
\label{ssec:Artin-groups-many-odd-labels}
By allying Theorem \ref{thm:Groups-with-infinite-cyclic-abelianisation}
and Proposition \ref{prp:Two-generator-Artin-groups-with-odd-label}
one can show 
that the derived groups of many Artin groups are finitely generated. 
In order to describe these groups, 
I need a definition.

Let $(\Gamma, \Delta(S, E, \lambda))$ be an Artin system.
The function $\lambda \colon E \to \left(\N \setminus \{0,1\}\right)$
allows one to divide the set of edges $E$ into two disjoint subsets,
defined by
\begin{equation}
\label{eq:Definition-Eodd-Eeven}
E_{\odd} =  \{e \in E \mid \lambda (e) \in 2\N + 1\}
\quad \text {and }\quad
E_{\even} =  \{e \in E \mid \lambda (e) \in 2\N\}.
\end{equation}

The announced result is now this:
\begin{prp}
\label{prp:Artin-systems-with-infinite-cyclic-abelianisation}
Let $\Gamma, \Delta(S, E, \lambda))$ is an Artin system
and assume the subgraph $(S, E_{\odd})$ is connected.
Choose a spanning tree $E_0$ of $\Delta$ 
inside the subgraph $(S, E_{\odd})$
and set
\begin{equation}
\label{eq:Bound-for-number-of-generators}
f = \sum\nolimits_{e \in E_0} \left( \lambda(e) -1\right).
\end{equation}
Then the derived group of $\Gamma$ is generated by at most $f$ elements 
and thus the metabelian top of $\Gamma$ is polycyclic.
\end{prp}

\begin{proof}
The defining relations of $\Gamma$ imply
that the function $s \mapsto 1$ extends to an epimorphism
$\pi \colon \Gamma \epi \Z$.
Moreover,
the assumption that the subgraph $(S, E_{\odd})$ is connected implies
that the abelianisation of $\Gamma$ is infinite cyclic;
it follows
that the kernel $N$ of $\pi$ coincides with the derived group of $\Gamma$.

Consider now the Artin system $(\Gamma_0, S)$ 
defined by the subgraph $(S, E_0)$ 
and the restriction $\lambda_0$ of $\lambda$ to $E_0$.
This subgraph is, by hypothesis, a spanning tree of $(S, E)$.
Fix an edge $e = \{s, s'\} \in E_{\odd}$ 
and let $\Gamma_e$ be the subgroup generated 
by the end points of the edge $e$.
Proposition \ref{prp:Two-generator-Artin-groups-with-odd-label}
then tells one 
that the kernel $N_e$ of the restriction of $\pi$ to $\Gamma_e$ 
is generated by $\lambda(e) - 1$ elements.  
Theorem \ref{thm:Groups-with-infinite-cyclic-abelianisation},
on the other hand, 
shows that $N = \ker \pi$ is generated by the union of subgroups $N_e$.
It follows, first of all, 
that $N$ is generated by at most $f = \sum\nolimits_{e \in E_0} (\lambda(e) -1)$ elements and, then,
that the metabelian top of $\Gamma_0$ is polycyclic.

Now to the Artin system $(\Gamma, (S, E, \lambda))$. 
The group $\Gamma$ has a presentation with the same generating set $S$ 
as has $\Gamma_0$
and a set of relations which includes those of $\Gamma_0$.
So $\Gamma$ is a quotient group of $\Gamma_0$
and the claim of the proposition follows.
\end{proof}

\begin{remarks}
\label{remarks:Artin-system}
1) Let $(\Gamma, \Delta(S, E, \lambda))$ be an Artin system
and let $\Delta_{\odd}$ be the subgraph of $\Delta(S, E, \lambda)$ 
with vertex set $S$ and with $E_{\odd}$ as set of edges. 
The graph $\Delta_{\odd}$ need not be connected;
its connected components correspond bijectively to a basis of the abelianisation
$\Gamma_{\ab}$ of $\Gamma$.
The Artin systems discussed in the previous proposition 
are thus precisely the \emph{Artin systems with abelianisations of rank 1}.

2) If one strengthens the hypotheses of Proposition 
\ref{prp:Artin-systems-with-infinite-cyclic-abelianisation}
suitably, the derived group of $\Gamma$ will be perfect; 
see Theorem \ref{thm:Artin-systems-with-perfect-derived-groups}.
\end{remarks}
In the proof of Proposition
\ref{prp:Artin-systems-with-infinite-cyclic-abelianisation} 
the assumption 
that the subgraph $(S, E_{\odd})$ be connected 
is used in two places:
it implies 
that the kernel of $\pi$ coincides with the derived group of 
$\Gamma$
and it guarantees 
that the kernels of the epimorphisms $\pi_e \colon G_e \epi \Z$  are finitely generated.
If the label of an edge is even and greater than 2, 
the latter property does not hold, 
as a look at Proposition
\ref{prp:Two-generator-Artin-groups-with-even-label}
will disclose, 
but is true if the label is 2 
since $G_e$ is then free abelian of rank 2.
The given proof of Proposition 
\ref{prp:Artin-systems-with-infinite-cyclic-abelianisation} 
allows one therefore to infer the following variant of the previous proposition:
\begin{prp}
\label{prp:Artin-systems-with-infinite-cyclic-quotient}
Let $(\Gamma, \Delta(S, E, \lambda))$ be an Artin system
and let $(S, E_{(2, \odd})$ be the subgraph of 
$(\Gamma, \Delta(S, E, \lambda))$ with vertex set $S$ 
and set of edges all those edges with label either 2 or an odd number greater than 2.
Assume this subgraph is connected;
choose a spanning tree $E_0$ inside  it
and set
\begin{equation}
\label{eq:Bound-for-number-of-generators}
f = \sum\nolimits_{e \in E_0} \left( \lambda(e) -1\right).
\end{equation}
Let $\pi \colon \Gamma \epi \Z$ denote the epimorphism
which sends each generator $s \in S$ to $1 \in \Z$.
Then the kernel of $\pi$ is generated by at most $f$ elements.
\end{prp}
\begin{remark}
\label{remark:Artin-systems-with-infinite-cyclic-quotient}
Suppose $(\Gamma, \Delta(S, E, \lambda))$ 
is an Artin system all whose edges have label 2.
\footnote{Recall that $E$ is not required to be the set of all edges of the complete graph on the set $S$}
Then this system is, by definition, a \emph{right angled Artin system}.
Proposition \ref{prp:Artin-systems-with-infinite-cyclic-quotient}
holds in this case, too, 
but a far more general result, due to J. Meier and L. vanWyck, is then available;
see \cite[Thm.\;6.1]{MeVa95}.
\end{remark}

%
\subsection{Generalized Artin systems}
\label{ssec:Generalised-Artin-systems}
Proposition 
\ref{prp:Artin-systems-with-infinite-cyclic-abelianisation}
shows 
that the derived groups of many Artin systems discussed in the literature
are finitely generated.
Its proof is based on two results,
the very general Theorem \ref{thm:Groups-with-infinite-cyclic-abelianisation}
and Proposition \ref{prp:Two-generator-Artin-groups-with-odd-label}
which guarantees that the derived groups of the two generator subgroups 
$\Gamma_e$ are finitely generated.
Now there are other two generator groups that enjoy this property 
and which can be used as building blocks for generalized Artin systems,
in particular the groups defined and discussed next.

\begin{definition}
\label{definition:Generalised-two-generator-Artin-systems}
Let $v_1$ and $v_2$ be \emph{positive} words in the generators $s$ and $s'$ 
that have the same odd length $\ell = 2m + 1$.
Assume that the initial letters of $v_1$, $v_2$ are distinct
and that their terminal letters are likewise distinct.
Suppose, in addition,
that $s$ occurs $m+1$ times in $v_1$ 
and that $s'$ occurs $m+1$ times in $v_2$.
Define $\Gamma(s, s', v_1, v_2)$  to be the group with generators $s$, $s'$
and defining relation $v_1 = v_2$.
\end{definition}
\begin{lem}
\label{lem:Generalized-Artin-groups}
Let $\Gamma = \Gamma(s, s', v_1, v_2)$ be the group defined before. 
The assignments $s \mapsto 1$ and $s' \mapsto 1$ extend then to an epimorphism  $\pi \colon \Gamma \epi \Z$.
Its kernel is a finitely generated free group 
which coincides with the derived group of $\Gamma$.
\end{lem}

\begin{proof}
Since the first letters of $v_1$ and $v_2$ are distinct, 
one may assume 
that the first letter of $v_1$ is $s$.
The words $v_1$ and $v_2$ have then the form

\[
v_1 = s^{a_1} (s')^{b_1} s^{a_2} \cdots 
\quad \text{ and } \quad
v_2 = (s')^{c_1} s^{d_1} (s')^{c_2} \cdots. 
\]
In the sequel, 
I shall view $\{s, s'\}$ as the generating set of a free group $F$ of rank 2
and $\pi$ as the epimorphism of $F$ onto $\Z$
which sends both $s$ and $s'$ to $1 \in \Z$.

The kernel of $\pi$ contains the element $u = s' \cdot s^{-1}$
and it is freely generated by the conjugates  $u_n = s^n u s^{-n}$
of  $u$ by the powers of $s$.
Since $s' = u \cdot s$ the words $v_1$ and $v_2$, 
when rewritten as words in $s$ and $u$,
have the form
\[
w_1 = s^{a_1} (u s)^{b_1} s^{a_2} \cdots 
\quad \text{ and } \quad
w_2 = (us)^{c_1} s^{d_1} (us)^{c_2} \cdots. 
\]
Two cases now arise, depending on 
whether the word $v_1$ ends in $s$ or in $s'$.
Assume first that the last letter of $v_1$ is $s$. 
Then
\begin{align*}
w_1 &= u_{a_1} u_{a_1 + 1} \cdots u_{a_1 + b_1 - 1} 
\cdot u_{a_1 + b_1 + a_2}\cdots u_{\ell -a_{h+1}-1} \cdot s^\ell, \\
w_2 &= u_{0} u_{1} \cdots u_{c_1-1} \cdot u_{c_1 + d_1}\cdots u_{\ell-1} \cdot s^{\ell}
\end{align*}
The crucial observations to be made at this point are these:
the indices of the generators $u_k$ occurring in the words $w_1$ or $w_2$
are strictly increasing;
moreover, the letter $u_0$ occurs in $w_2$, but not in $w_1$;
similarly, the letter $u_{\ell-1}$ occurs in $w_2$,
but no letter with a higher index does, 
while the highest index of a letter $u_k$ present in $w_1$ is
$\ell - a_{h+1}-1 < \ell-1$.
The kernel of $\pi$ is therefore a free group of rank $\ell -1$.

The analysis of the case 
where the word $v_1$ ends in $s'$ 
is quite similar.
\end{proof}

Upon combining the preceding lemma with
Theorem \ref{thm:Groups-with-infinite-cyclic-abelianisation} one arrives at
\begin{prp}
\label{prp:Generalized-Artin-groups-with-polycyclic-metabelian-tops}
Let $\Delta$ be a finite, connected, combinatorial graph, 
with non-empty set of vertices $S$
and set of edges  
$E \subset \{\{s,s'\} \mid s \in S, s' \in S \text{ and } s \neq s'  \}$.
Choose a total order  $<$ on $S$
and select, 
for every ordered pair $e = (s,s')$ with elements in $S$ and $s < s'$,
a group $\Gamma(e, v_1(e), v_2(e))$, 
as described in Definition 
\ref{definition:Generalised-two-generator-Artin-systems}. 
Finally, 
define
\[
\Gamma = \Gamma(\Delta, E, \{\Gamma(e,v_1(e), v_2(e)) \mid e \in E \}
\]
to be the group of the graph of groups 
with vertex groups the infinite cyclic groups $G_s$
generated by the elements $s \in S$
and edge groups the groups $\Gamma(e, v_1(e), v_2(e))$.
Then the derived group of $\Gamma$ is finitely generated
and hence the metabelian top of $\Gamma$ is polycyclic.
\end{prp}

\begin{question}
\label{question:Interest-of-generalisation}
Have the generalized Artin groups described in
Proposition \ref{prp:Generalized-Artin-groups-with-polycyclic-metabelian-tops} ever been considered in the literature? 
Do they have interesting applications or special properties?
\end{question}
\section{Artin systems of finite type}
\label{sec:Artin-systems-of-finite-type}
In this section,
I study the metabelian tops of a particular class of Artin systems,
the class of \emph{Artin systems of finite type}. 
My results are not new,
they reproduce part of an earlier investigation, 
undertaken by Jamie Mulholland and Dale Rolfsen in \cite{MuRo06}.
The proofs employed in the two studies differ, however,
as I use \emph{Homology Theory of Groups}, 
and not the method of Reidemeister-Schreier, 
to determine the abelian group $\Gamma'_{\ab}$ 
of an irreducible Artin group of finite type.
The chosen approach allows one, in addition,
to investigate the metabelian tops of some Artin systems 
whose associated Coxeter groups are infinite.
%
\subsection{Artin systems of finite type -- basics}
\label{ssec:Artin-groups-of-finite-type-basics}
%
I begin by recalling the definition of a 
\emph{Coxeter system associated to an Artin system}.
Let $(\Gamma, \Delta(S, E, \lambda))$ be an Artin system
and let $R$ be its standard set of defining relations $ss' \cdots = s's \cdots$.

The Coxeter system $(W(\Gamma), S)$ 
associated to $(\Gamma, \Delta(S, E, \lambda))$ consists of the Coxeter group 
with presentation  
\begin{equation}
\label{eq:Associated-Coxeter-group}
W(\Gamma) = \langle S \mid R \cup \{s^2 = 1 \mid s \in S \}\; \rangle,
\end{equation}
and of the set of generators $S$.

Artin systems of finite type are now defined as follows:
\begin{definition}
\label{def:Artin-group-of-finite-type}
An Artin system $(\Gamma, \Delta(S, E, \lambda))$ is said to of \emph{finite} 
(or \emph{spherical}) type
if its associated Coxeter group is finite.
\end{definition}

Every finite Coxeter system $(W, S)$ is the direct product 
$(W_1, S_1) \times \cdots \times (W_r, S_r)$ of Coxeter systems
which are \emph{irreducible}
in the sense 
that their Coxeter graphs are the connected components 
of the Coxeter graph of $(W, S)$.
The finite, irreducible Coxeter systems have been classified 
by H. S. M. Coxeter in \cite{Cox34};
their Coxeter graphs are displayed in Figure 
\ref{fig:Coxeter-graphs-of-the-irreducible-Artin-groups-of-finite-type}; 
\cf \cite[Theorem 2.7]{Hum90}.
\footnote{In this list, the Coxeter system of type $G_2$ is omitted,
but the isomorphic system of type $I_2(6)$ is included.}
\begin{figure}[ht!]
{\small
\vspace*{0.3cm}\hspace*{1.2cm} $A_\ell$
\hspace*{8.0cm} ($\ell \geq 1$ vertices) \par
\vspace*{-0.55cm} 
\begin{center}
\includegraphics[width=7cm]{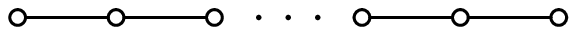}\par
\end{center}
\vspace*{0.3cm}\hspace*{1.2cm} $B_\ell$ 
\hspace*{8.0cm} ($\ell \geq 2$ vertices) \par
\vspace*{-0.5cm} 
\begin{center}
\includegraphics[width=7cm]{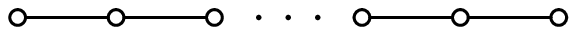}\par
\end{center}
\vspace*{-0.9cm}\hspace*{4.35cm} $4$ \par
\vspace*{0.95cm}\hspace*{1.3cm} $D_\ell$  
\hspace*{8.1cm} ($\ell \geq 4$ vertices) \par
\vspace*{-1.05cm}
\begin{center}
\includegraphics[width=6.6cm]{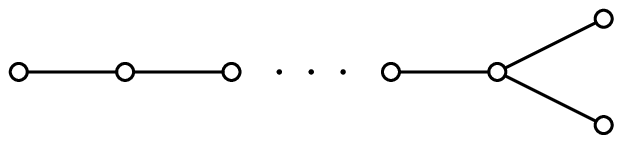}\par
\end{center}
\vspace*{0.2cm}\hspace*{1.2cm} $E_6$ \hspace*{10.2cm} 
\vspace*{-0.85cm}\par
\begin{center}
\includegraphics[width=5.2cm]{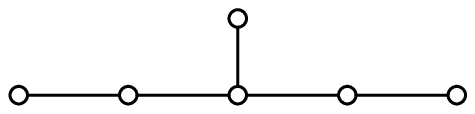}\hspace*{1.4cm}\par
\end{center}
\vspace*{0.4cm}\hspace*{1.2cm} $E_7$ \hspace*{10.2cm}\par
\vspace*{-0.85cm} 
\begin{center}
\includegraphics[width=6.4cm]{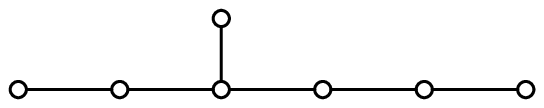}\hspace*{0.25cm}\par
\end{center}
\vspace*{0.4cm}\hspace*{1.2cm} $E_8$  \hspace*{10.2cm} \par
\vspace*{-0.85cm}
\begin{center}
\hspace*{0.9cm}\includegraphics[width=7.5cm]{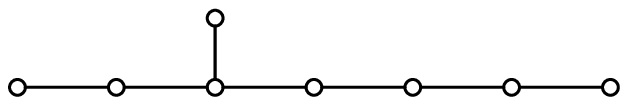} \hspace*{0.0cm}\par
\end{center}
\vspace*{0.65cm}\hspace*{1.2cm} $F_4$ \hspace*{10.2cm} \par
\vspace*{-0.55cm} 
\begin{center}
\includegraphics[width=4.2cm]{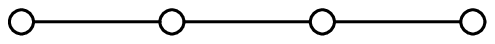}\hspace*{1mm}\par
\end{center}
\vspace*{-1cm}\hspace*{-0.1cm} $4$ \par
\vspace*{0.9cm}\hspace*{1.2cm} $H_3$  \hspace*{10.2cm}\par
\vspace*{-0.5cm} 
\begin{center}
\includegraphics[width=3.0cm]{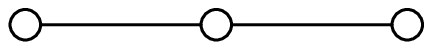}\par
\end{center}
\vspace*{-1cm}\hspace*{0.9cm} $5$ \par
\vspace*{0.9cm}\hspace*{1.2cm}  $H_4$ \hspace*{10.2cm}\par
\vspace*{-0.5cm} 
\begin{center}
\includegraphics[width=4.2cm]{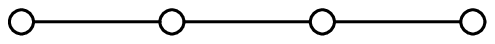}\par
\end{center}
\vspace*{-1cm}\hspace*{1.9cm} $5$ \par
\vspace*{0.9cm}\hspace*{1.2cm} $I_2(m)$\hspace*{10.2cm}\par
\vspace*{-0.5cm} 
\begin{center}
\includegraphics[width=2.0cm]{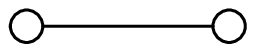}\par
\end{center}
\vspace*{-0.95cm}\hspace*{-0.05cm} $m$ \par
\vspace*{-0.1cm}
\hspace*{9.6cm} ($m \geq 5$) \par
\vspace*{-1.05cm}
}
\vspace*{1cm}
\caption{Coxeter graphs of the irreducible Artin systems of finite type}
\label{fig:Coxeter-graphs-of-the-irreducible-Artin-groups-of-finite-type}
\end{figure}

%
The abelianisation of such a Coxeter system is infinite cyclic,
except for the systems of type $B$, of type $F_4$ 
or of type $I_2(2n)$ with $n \geq1$
where the abelianisation is free abelian of rank 2.

The Coxeter graphs displayed in Figure
\ref{fig:Coxeter-graphs-of-the-irreducible-Artin-groups-of-finite-type}
are to be interpreted as follows.
\footnote{I follow the convention according to which 
commuting relations $s_i s_j = s_j s_i$ between generators in $S$ 
are not represented by edges of the Coxeter graph; 
\cf \cite[Prop.\;6.1]{Hum90}.}
Let $\Delta(S,E,\lambda)$ be the Coxeter graph of an irreducible Coxeter system of finite type,
with generating set $S$, edge set $E$ and labelling function 
$\lambda \colon E \to \N \smallsetminus \{0,1,2\}$.
Enumerate the elements of $S$, say $s_1$, \ldots, $s_\ell$.
The Artin system associated to $\Delta$ 
is the group with generating set $S$,
with set of edges 
$ E = \{ \{s_i, s_j\} \mid 1 \leq  i < j  \leq \ell\}$
and with the following set of 
$\tfrac{1}{2} (\card S - 1) \cdot \card S$ 
defining relations:
for every couple $(i, j)$ with $i < j$ and $\{i,j\} \in E$
there is a relation $u(s_i,s_j) = u(s_j,s_i)$
where $u(s_i,s_j) = s_is_js_i \cdots$ is the positive word of length 
$\lambda(e) \geq 3$ with first letter $s_i$ and in which the letters $s_i$ and $s_j$ alternate,
and for every couple $(i, j)$ with $i < j$ and $\{i,j\} \notin E$
there is the commuting relation $s_is_j= s_js_i$.

Each of the Coxeter graphs listed in Figure
\ref{fig:Coxeter-graphs-of-the-irreducible-Artin-groups-of-finite-type}
has a spanning tree the edges of which have labels greater than 2,
and these labels are all odd, 
except for the Coxeter graphs of type $B_\ell$, of type  $F_4$, 
or of type $I_2(m)$ with $m$ even.
\footnote{By convention,
an edge of the Coxeter graph with label 3 is not decorated by the number 3.}
The spanning tree is a line, except if the system is of type $D$ or of type $E$.
The abelianisations of the irreducible Artin groups of finite type 
are thus infinite cyclic or free abelian of rank 2, 
and the latter case happens precisely 
if the defining graph has an edge  with label 4 
or the the system is of type $I_2(2n)$ with $n \geq 1$.
%
%
\subsection{Summary of results}
\label{ssec:Summary-of-results}
Given an irreducible Artin system $(\Gamma, \Delta(S, E, \lambda))$ of finite type,
we are interested in the answers to the following questions:
\begin{questions}
\label{questions:derived-group}
\begin{enumerate}
\item [\textbf{Q1:}] 
What is the structure of $\Gamma'$?  Is it finitely generated?
\item [\textbf{Q2:}]
What is the structure of $\Gamma'_{\ab} = \Gamma' /\Gamma''$?
Is it finitely generated?
\item [\textbf{Q3:}]
Is $\Gamma'$ perfect?
\end{enumerate}
\end{questions}
The obtained answers are rather complete 
if the abelianisation of $\Gamma$ is infinite cyclic 
or if $\Gamma$ is of type $I_2$,
but less so if $\Gamma$ is of type $B$ or of  type $F_4$.

Here is a brief summary of the results that will be proved in this Section
\ref{sec:Artin-systems-of-finite-type}.
\begin{thm}
\label{thm:Answers-to-questions}
Let $(\Gamma, \Delta(S, E, \lambda))$ be an irreducible Artin system of finite type.
Then the following statements hold:
\begin{enumerate}[(i)]
\item $\Gamma'$ is finitely generated 
whenever $\Gamma_{\ab}$ its infinite cyclic.
\item $\Gamma'_{\ab}$ is free abelian and of finite rank 
unless $\Gamma$ has type $I_2(2n)$ with $n \geq 3$. 
\item  $\Gamma'$ is perfect unless $(\Gamma, \Delta(S, E, \lambda))$ 
has one of the types
\[
A_2, \quad A_3, \quad B_2, \quad B_3, \quad  B_4, \quad D_4, \quad F_4
\quad\text{or}\quad I_2(m)  \text{ with } m \geq 5.
\]
\end{enumerate}
\end{thm}

Claim (i) is a direct consequence of Proposition
\ref{prp:Artin-systems-with-infinite-cyclic-abelianisation}.
The remaining assertions will be established 
by computing the homology groups $H_1(\Gamma', \Z)$ 
for irreducible Artin systems of finite type with few generators
and by using an inductive argument for the remaining Artin systems.
Note that, by Shapiro's Lemma 
\footnote{see, \eg \cite[Ch.\;III, Prop.\;6.2]{Bro94}}, 
the group $H_1(\Gamma', \Z)$ is isomorphic to the homology group
$H_1(\Gamma, \Z{\Gamma_{\ab}})$
and that the latter group can be computed with the beginning of the 
$\Z{\Gamma}$-free resolution of $\Z$ 
associated to the standard presentation of the Artin group $\Gamma$.
%
%
\subsection{Study of the differentials $d_1$ and $d_2$}
\label{ssec:Study-of-differentials}
%
Let $\ell$ be the number of generators of $\Gamma$ 
in its standard presentation.
The beginning of the previously mentioned $\Z{\Gamma}$-free resolution of $\Z$ 
has the form
\begin{equation}
\label{eq:Beginning-of-resolution}
\cdots \longrightarrow 
\bigoplus\nolimits_{(i,j)} \Z{\Gamma} \cdot b_{(i,j)}
 \xrightarrow{\delta_2}
\bigoplus_{1 \leq k \leq \ell} \Z{\Gamma} \cdot b_k
 \xrightarrow{\delta_1}\Z{\Gamma} \to \Z \to 0.
\end{equation}
In the above, 
$b_{(i,j)}$ and $b_k$ denote the basis elements of the free modules of the resolution
in dimensions 2 and 1, respectively.
To simplify notation,
set $Q =\Gamma_{\ab}$ and $R = \Z{Q}$,
and put  $d_2 = \id_R \otimes_{R}\delta_2$ 
and $d_1 = \id_R \otimes_{R}\delta_1$.
The group $H_1(\Gamma, \Z{\Gamma_{\ab}})$ is then the homology group 
at the middle term of the complex
\begin{equation}
\label{eq:Complex-for-homology}
\bigoplus\nolimits_{(i,j)} R \cdot b_{(i,j)} 
 \xrightarrow{d_2}
\bigoplus_{1 \leq k \leq \ell} R \cdot b_k 
 \xrightarrow{d_1} R
\end{equation}

\subsubsection{Differential $d_1$ and its kernel}
\label{sssec:d1-and-ker-d1}
Two cases arise, 
depending on whether $Q$ is infinite cyclic or free abelian of rank 2.
In the first case,
the group $Q$ is generated by the common image, say $s$,  
of the standard generators $s_1$,  \ldots, $s_\ell$  in $Q$. 
The differential $d_1$ is then given by multiplication by the column vector 
\[
\left(1-s, 1-s, \ldots, 1-s\right)
\]
and an easy calculation shows that
\begin{equation}
\label{eq:kernel-d1-rank-1}
\ker d_1 
= 
R \cdot (b_1 - b_2)  
\oplus \cdots \oplus 
R\cdot (b_{\ell-1} - b_\ell).
\end{equation}

Consider next the case 
where $Q = \Gamma_{\ab}$ is free abelian of rank 2.
The system $(\Gamma, \Delta(S, E, \lambda))$ 
is then either of type $B_\ell$ for some $\ell \geq 2$,
of type $F_4$, or of type $I_2(2n)$ for some $n \geq 3$.
If it is of type $B_\ell$ 
the standard generators $s_1$, \ldots, $s_{\ell-1}$ map onto one and the same element, say $s$, 
and $e_\ell$ maps onto an element, say $t$, 
all in such a way 
that $s$ and $t$ are free generators of the free abelian group $Q$.
The differential $d_1$ is given by multiplication by the column vector 
$\left(1-s, \ldots, 1-s, 1-t \right)$
and
\begin{equation}
\label{eq:kernel-d1-type-B}
\ker d_1 
= 
R \cdot (b_1 - b_2)  \oplus \cdots 
\oplus 
R\cdot (b_{\ell-2}-b_{\ell-1}) 
\oplus 
R\cdot ((1-t)b_{\ell-1} + (s-1)b_\ell),
\end{equation}
as a short calculation will confirm.
The case of a system of type $F_4$ is similar:
the generators $s_1$ and $s_2$ map onto the same element, say $s$, of $Q$,
the generators  $s_3$ and $s_4$ go to the same element $t$, say, 
and $s$, $t$ form a basis of $Q$.
The differential $d_1$ is given by multiplication by the column vector 
$\left(1-s,1-s,  1-t, 1-t \right)$
and
\begin{equation}
\label{eq:kernel-delta1-type-F4}
\ker d_1 
= 
R \cdot (b_1 - b_2)  
\oplus 
R\cdot ((1-t) b_2 + (s-1)b_3) 
\oplus 
R\cdot (b_3 - b_4).  
\end{equation}
If, finally, $\Gamma$ is of type $I_2(2n)$ we are in the case already treated 
in section \ref{ssec:Two-generator-Artin-systems}. 
\subsubsection{Differential $d_2$ and its image}
\label{sssec:d2-and-im-d2}
The differential $d_2$ is given by a matrix with $\tfrac{1}{2}(\ell-1)\ell$ rows,
one row for each relation in the standard presentation,
and $\ell$ columns.
To find the rows of the matrix, 
one computes the partial derivatives of the defining relators of the system in question.
The standard relations of every irreducible Artin system of finite type, 
distinct from those of type $I_2$, 
have one of 4 forms, namely
\begin{equation}
\label{eq:List-relations}
\left.
\begin{aligned}
s_i s_j &= s_j s_i, 
&
s_i s_js_i &= s_j s_i s_j, \\
(s_i s_j)^2 &= (s_j s_i)^2, 
&
(s_i s_j)^2 s_i &= s_j (s_i  s_j )^2.
\end{aligned}
\hspace*{2mm}\right \rbrace
\end{equation}
Corresponding relators are
\begin{equation}
\label{eq:List-relators}
\left.
\begin{aligned}
r_2 &= s_i s_j \cdot s_i^{-1} s_j^{-1}, &
r_3 &= s_i s_j s_i \cdot  s_j^{-1} s_i^{-1}s_j^{-1}, \\
r_4 &=(s_i s_j)^2 \cdot  (s_i^{-1} s_j^{-1})^2, &
r_5 &= (s_i s_j)^2 s_i \cdot  s_j^{-1} (s_i^{-1}  s_j^{-1})^2.
\end{aligned}
\hspace*{2mm}\right \rbrace
\end{equation}
The non-zero partial derivatives of the relators $r_2$ and $r_3$ are
\begin{align*}
D_{s_i}(r_2)  &= 1 - s_i s_j s_i^{-1}, 
&
D_{s_j} (r_2) &= s_i - r_2, \\
D_{s_i}(r_3)  &= 1 + s_i s_j - r_3 s_j, 
&
D_{s_j} (r_3) &= s_i - s_i s_j s_i s_j^{-1} - r_3, 
\end{align*}
while those of the relators $r_4$ and $r_5$ are
\begin{align*}
D_{s_i}(r_4)  &= 1 + s_i s_j - (s_i s_j)^2 s_i^{-1} - r_4 s_j,\\
D_{s_j} (r_4) &= s_i +s_i s_j s_i - (s_is_j)^2s_i^{-1} s_j^{-1} - r_4,\\
D_{s_i}(r_5)  
&= 
1 + s_i s_j + (s_is_j)^2
- (s_is_j)^2 s_i \cdot s_j^{-1} s_i^{-1} - r_5s_j ,\\
D_{s_j} (r_5) 
&=
 s_i +s_i s_j s_i - (s_is_j)^2s_i \cdot s_j^{-1} - 
 (s_is_j)^2s_i\cdot  s_j^{-1}s_i^{-1} s_j^{-1}  -r_5.
\end{align*}

Now it is not these partial derivatives
that enter into the description of $d_2$,
but there images under the canonical map 
$\widehat{\phantom{-}} \colon \Z{\Gamma} \epi \Z{\Gamma_{\ab}}$.
Two cases arise, 
depending on whether 
the length of the relator is twice an odd number or twice an even number.
In the first case, 
the generators $s_i$ and $s_j$ map onto the same element $s$, say, 
of $Q = \Gamma_{\ab}$
and the corresponding row in the matrix $M$  is
\begin{align}
&(1 - s + s^2) (b_i - b_j)&  \text{ for } &r_3,
\label{eq:Row-for-r3}\\
&(1 - s + s^2 -s^3 + s^4) (b_i - b_j) &\text{ for } &r_5.
\label{eq:Row-for-r5}
\end{align}

Consider next the relator $r_2$.
Then $s_i$ and $s_j$  can map onto the same element 
or onto $\Z$-linearly independent elements, 
depending on the type of the system.
If $\Gamma_{\ab}$ is infinite cyclic, 
they map onto the same element
and the corresponding row of the matrix $M$ is
\begin{equation}
\label{eq:Relator-r2-rank-1}
(1 - s) (b_i - b_j).
\end{equation}
If, on the other hand, $\Gamma_{\ab}$ is free abelian of rank 2,
the generators can map onto the same element,
in which case the corresponding row will have the form
\eqref{eq:Relator-r2-rank-1}.
If the generators map onto linearly independent elements,
say $s_i$ has image $s$ and $s_j$ has image $t$,
the row in the matrix $M$ will be 
\begin{equation}
\label{eq:Relator-r2-rank-2}
(1 - t) b_i  + (s-1) b_j .
\end{equation}

Assume, finally, that the relator is $r_4$.
Then $\Gamma$ is either of type $B_\ell$ for some $\ell  \geq 2$ 
or of type $F_4$ or of type $I_2(4)$.
In all three cases, the generators $s_i$ and $s_j$ map onto linearly independent elements of $Q$ 
and the associated row in the matrix $M$ is
\begin{equation}
\label{eq:Row-for-r4}
(1 + st) \left((1-t)b_i  + (s-1) b_j \right).
\end{equation}
%
\subsection{Artin systems with 4 generators}
\label{ssec:Artin-systems-with-4-generators}
%
Suppose  $(\Gamma, \Delta(S, E, \lambda))$ is an Artin system of finite type 
with \emph{infinite cyclic abelianisation}.
One goal of this Section 
\ref{sec:Artin-systems-of-finite-type}
is to show 
that the derived group of $\Gamma$ is perfect,
unless the system belongs to a small group of exceptions.
These exceptions arise only 
if the number of standard generators is small,
the exact bound depending on the type.

In this section, 
Artin systems with 4 generators will be investigated;
later on, the obtained results will be used in an inductive argument.
I begin with a system of type $A_4$ 
and then pass on to a generalization of this system,
a generalization which covers, in particular, the system of type $H_4$.
%
\subsubsection{Artin system of type $A_4$}
\label{sssec:Artin-system-of-type-A4}
%
Let $(\Gamma, \Delta(S, E, \lambda))$ be an Artin system of type $A_4$ 
with standard generating set $S = \{s_1, s_2, s_3, s_4\}$.
The group has 6 defining relations,
the three commuting relations
\begin{equation}
\label{eq:Commuting-relations-A4}
s_1 s_3 = s_3 s_1, \quad
s_2 s_4 = s_4 s_2, \quad
s_1 s_4 = s_4 s_1,
\end{equation}
and the three relations
\begin{equation}
\label{eq:Other-relations-A4}
s_1 s_2 s_1 = s_2 s_1 s_2, \quad
s_2 s_3 s_2 = s_3 s_2 s_3, \quad
s_3 s_4 s_3= s_4 s_3s_4.
\end{equation}
According to the plan laid out in section 
\ref{ssec:Artin-groups-of-finite-type-basics},
we shall compute the homology group $\ker d_1/\im d_2$.
In order to describe the kernel of $d_1$ and the image of $d_2$ succinctly,
set $Q =\Gamma_{\ab}$ and $R = \Z{Q}$. 
The generators $s_i$ in $S$ map to the same element, say $s$, of $Q$ 
and
\begin{equation}
\label{eq:ker-d1-type-A4}
\ker d_1 
= 
R \cdot (b_1- b_2)  \oplus R \cdot (b_2- b_3)  \oplus R \cdot (b_3- b_4) 
\end{equation}
by formula \eqref{eq:kernel-d1-rank-1}.
So $\ker d_1$ is a free $R$-module with basis 
\begin{equation}
\label{eq:Introducing-vectors-u-v-w}
u = b_1-b_2, \quad v = b_2 - b_3 \quad \text{and}\quad w = b_3 - b_4.
\end{equation}
The differential $d_2$ is given by a matrix $M(A_4)$ with entries in $R$;
it has 6 rows, one row for each defining relator, and 4 columns.
By equation \eqref{eq:Relator-r2-rank-1},
the rows corresponding to the relations listed in
\eqref{eq:Commuting-relations-A4}
are
\begin{equation}
\label{eq:Rows-L1-L2-L3-of-M(A4)}
L_1 = (1-s) (b_1 - b_3),\hspace*{2mm}
L_2 = (1-s) (b_2 - b_4),\hspace*{2mm}
L_3 = (1-s)(b_1-b_4).
\end{equation}
Now, 
one is not interested in the individual rows of the matrix $M(A_4)$,
but in the $\Z{Q}$-linear combination of these rows.
Here a surprise happens: 
\begin{equation}
\label{eq:Three-linear-combinations-A4}
L_3 - L_2  = (1-s) u, \quad 
L_1 + L_2 - L_3 = (1-s) v, \quad
L_3 -L_1 = (1-s)w . 
\end{equation}
These three rows are the rows
one would obtain if the commutators 
$[s_1,s_2]$, $[s_2, s_3]$ and $[s_3, s_4]$ 
were relators of the group $\Gamma$.

Consider next the three relations displayed in equation 
\eqref{eq:Other-relations-A4}.
In view of equation \eqref{eq:Row-for-r3},
they give rise to the rows
\begin{equation}
\label{eq:Rows-4to6-of-M(A4)}
L_4 = (1-s + s^2) u, \quad
L_5 = (1-s + s^2) v, \quad
L_6 = (1-s + s^2) w.
\end{equation}
The image of $d_2$ contains therefore the linear combinations 
\begin{equation}
\label{eq:Generators-im-d2-for-A4}
L_4 +  s \cdot (L_3-L_2) = u,
\quad
L_5 +s \cdot (L_1 + L_2 - L_3) = v,
\text{ and } 
L_6 + s \cdot (L_3 -L_1 )= w.
\end{equation}
A comparison of equations 
\eqref{eq:ker-d1-type-A4}, \eqref{eq:Introducing-vectors-u-v-w} 
and \eqref{eq:Generators-im-d2-for-A4} 
finally shows that $\im d_2 = \ker d_1$.
The homology group $H_1(\Gamma, \Z{\Gamma_{\ab}})$ is thus trivial;
equivalently, $\Gamma'$ is perfect.
%
\subsubsection{More Artin systems with 4 generators}
\label{sssec:More-Artin-systems-with-4-generators}
%
If one scrutinizes the calculation of the the homology group $\ker d_1/\im d_2$,
carried out in the preceding section,
one detects that this calculation can be adapted
to other Artin systems with four generators,
provided the labels of the edges 
$\{s_1, s_2\}$, $\{s_2, s_3\}$ and $\{s_3, s_4\}$ 
are odd numbers greater than 2.
Indeed, 
let $F$ denote the free group on $S = \{s_1, s_2, s_3, s_4\}$
and assume the Artin system $(\Gamma, S)$ 
has the following six relations:
the three commuting relations \eqref{eq:Commuting-relations-A4}
and the relations
\begin{equation}
\label{eq:Other-relations-new-Artin-groups}
(s_1 s_2)^{m} s_1= s_2 (s_1 s_2)^{m}, \quad
(s_2 s_3)^{n} s_2= s_3 (s_2 s_3)^{n}, \quad
(s_3 s_4)^{p} s_3= s_4 (s_3 s_4)^{p},
\end{equation}
where $m$, $n$ and $p$ are positive natural numbers.
The abelianisation of $\Gamma$ is infinite cyclic 
and the generators $S$ map to the same element of 
$Q = \Gamma_{\ab}$; call it $s$.  
The three commuting relations imply,
as in section \ref{sssec:Artin-system-of-type-A4},
that the image of $d_2 \colon \Z{Q}^6 \to \Z{Q}^4$ 
contains the vectors
\[
v_{12} =(1-s) (b_1- b_2), \quad v_{23} =(1-s) (b_2-b_3), \quad 
v_{34}=(1-s)(b_3-b_4).
\]

Consider now one of the relations in equation 
\eqref{eq:Other-relations-new-Artin-groups}, 
say the first one. 
Set 
\[
U = (s_1s_2)^m s_1\quad\text{and}\quad V =s_2(s_1s_2)^m,
\]
and denote the corresponding relator $U \cdot V^{-1}$ by $r$.
If $D$ is one of the derivations $D_{s_1}$ or $D_{s_2}$, 
the following formulae hold in the group ring of the free group $F$
on $\{s_1, s_2 \}$:
\begin{align}
D(U \cdot V^{-1}) &= D(U) + U\cdot D(V^{-1}), \quad
D(V^{-1}) = - V^{-1} D(V), \quad \text{and so} \notag
\\
D(U \cdot V^{-1}) &= D(U) - U V^{-1} \cdot D(V)
\label{eq:Identity-for-derivations}
\end{align}
(see, \eg \cite[p.\;549, formulae (1.2') and 1.6)]{Fox53}).
It follows that
\begin{align*}
D_{s_1}(U) &= 1 + s_1s_2 + \cdots + (s_1s_2)^m,\\
D_{s_2}(U)&= s_1 + s_1s_2s_1 + \cdots + (s_1 s_2)^{m-1}s_1, \\
D_{s_1}(V) &=s_2 + s_2 s_1 s_2 + \cdots + (s_2 s_1)^{m-1} s_2,\\
D_{s_2}(V) &= 1 + s_2 s_1 + \cdots + (s_2 s_1)^m .
\end{align*}
By using the identity \eqref{eq:Identity-for-derivations}
and the above calculations,
one can find the images of the elements $D_{s_i}(U\cdot V^{-1})$
under the canonical map $\ab \colon \Z{F} \to \Z{\Gamma_{\ab}}$.
Since the generators $s_k$  map to the same element $s \in \Gamma_{\ab}$,
these images are:
\begin{align*}
\widehat{D_{s_1}}(U \cdot V^{-1}) 
&= 
\widehat{D_{s_1}}(U) - \widehat{D_{s_1}} (V)\\
&=
(1 + s^2 + \cdots +  s^{2m}) - (s + s^3 + \cdots + s^{2m-1}), 
\quad \text{and}\\
\widehat{D_{s_2}}(U \cdot V^{-1}) 
&= 
\widehat{D_{s_2}}(U) - \widehat{D_{s_2}} (V)\\
&=
(s + s^3 + \cdots + s^{2m-1}) - ( 1 + s^2 + \cdots + s^{2m}).
\end{align*}
The row $L_4$ corresponding to the relator $r$ is therefore
\begin{equation}
\label{eq:L4-for-general-Artin-group}
L_4 = (1 - s + s^2 +\cdots  +  (-s^{2m-1}) + s^{2m}) \cdot (b_1 -b_2)
\end{equation}
and so the image  of $d_2$ contains the linear combination 
\[
L_4 + (s + s^3 + \cdots + s^{2m-1})\cdot (1-s) (b_1-b_2) = b_1 - b_2.
\]
The situation is quite analogous for the rows $L_2$, $L_5$ 
and for the rows $L_3$, $L_6$.
The above calculations and Proposition 
\ref{prp:Artin-systems-with-infinite-cyclic-abelianisation}
therefore establish 
\begin{thm}
\label{thm:More-groups-with-perfect-derived-group}
Let $(\Gamma, \Delta(S, E, \lambda))$ be an Artin system 
with standard generating set  $S = \{s_1, \ldots, s_4\}$,
with commuting relations 
\eqref{eq:Commuting-relations-A4}
and relations
\eqref{eq:Other-relations-new-Artin-groups}.
Then the derived group of $\Gamma$ is finitely generated and perfect.
\end{thm}

As will be seen in the sequel,
Theorem \ref{thm:More-groups-with-perfect-derived-group} 
allows one to deduce that the derived group of many Artin groups
with infinite cyclic abelianisations is perfect.
%
\subsection{Artin systems of type $A$}
\label{ssec:Artin-systems-of-type-A}
%
Let $(\Gamma, \Delta(S, E, \lambda))$ 
be an Artin system of type $A_\ell$ for some integer $\ell \geq 1$.
If $\ell =1$, the group $\Gamma$ is infinite cyclic; 
so its derived group is trivial and hence perfect. 
If $\ell = 2$,
the system is one of the systems treated in Proposition 
\ref{prp:Two-generator-Artin-groups-with-odd-label}
and its derived group is free of rank 2.
The Artin system of type $A_4$ has been dealt with in section
\ref{sssec:Artin-system-of-type-A4}; 
those with $\ell = 3$
or $\ell > 4$ will be discussed in the next two sections.
%
\subsubsection{Artin system of type $A_3$}
\label{sssec:Artin-system-of-type-A3}
%
The group $\Gamma$ has the generators $s_1$, $s_2$, $s_3$
and the defining relators
\begin{equation}
\label{eq:Relators-A3}
r_2 = s_1 s_3 \cdot s_1^{-1} s_3^{-1},
\quad
r_3 = s_1 s_2 s_1 \cdot  s_2^{-1} s_1^{-1}s_2^{-1},
\quad
\tilde{r}_3 = s_2 s_3 s_2 \cdot  s_3^{-1} s_2^{-1}s_3^{-1}.
\end{equation}
The kernel of $d_1$ is a free submodule of 
$R b_1 \oplus R  b_2 \oplus R  b_3$,
namely
\begin{equation}
\label{eq:Kernel-d1-type-A4}
\ker d_1 = R \cdot (b_1 - b_2) \oplus  R \cdot (b_2 - b_3)
\end{equation}
(see formula \eqref{eq:kernel-d1-rank-1}).
Set $u = b_1 - b_2$ and $v = b_2-b_3$.
In view of equations \eqref{eq:Row-for-r3} and \eqref{eq:Relator-r2-rank-1},
the image of the differential $d_2$ is then generated by the three vectors
\[
x = (1 - s + s^2)\cdot u,
\quad
y=  (1 - s + s^2)\cdot v,
\quad
z= (1-s) (u + v).
\]
Since $x + y = (1-s+s^2) (u + v)  = (u+v) - s \cdot z$, 
one has $u + v \in \im d_2$ and so 
$ \im d_2 = R \cdot (u+v) + R \cdot  (1 - s + s^2)u$.
It follows that the homology group $\ker d_1 / \im d_2$ 
is isomorphic to  $R / R(1 - s + s^2)$ 
and hence free abelian of rank 2.
%
%
\subsubsection{A generalization}
\label{sssec:Generalization}
%
Let $(\Gamma, \Delta(S, E, \lambda))$ be an Artin system of type $A_\ell$ 
with standard generating set $S = \{s_1, s_2, \ldots, s_\ell\}$,
and assume that $\ell \geq 4$.
The aim is to show that $\Gamma'$ is perfect.
This goal can be reached by (at least) two different roads.
One can imitate the proof given in section \ref{sssec:Artin-system-of-type-A4} 
but with $\ell$ instead of 4 generators $s_k$.
Alternatively,
one can deduce the claim from Theorem 
\ref{thm:Artin-systems-with-perfect-derived-groups}
below.

I shall travel on the second road
and use, in my proof,  
a distance function $d$ on a spanning tree $T$ of a combinatorial graph,
defined as follows: let $v_1$ and $v_2$ be vertices of $T$ and 
count the number of edges of the geodesic path from $s_1$ to $s_2$.
The distance is then given by the formula
\begin{equation}
\label{eq:Definition-distance}
d(s_1, s_2 ) = \text{ number of edges of the geodesic path from $s_1$ to $s_2$}.
\end{equation}

\begin{thm}
\label{thm:Artin-systems-with-perfect-derived-groups}
Let $(\Gamma, \Delta(S, E, \lambda))$ be an Artin system
and assume the following hypotheses are satisfied:
\begin{enumerate}[(i)]
\item the graph $(S, E)$ is connected and admits a spanning tree $T$ 
all whose edges are labelled by odd numbers;
\item if $s$, $s'$ are vertices of $T$ with distance 2,
the commuting relation $s s' = s' s$ holds in $\Gamma$;
\item there exists an edge $e_0 = \{s_1, s_2\}$ in $T$ such 
that the difference $u = b_1 - b_2$ of the basis vectors $b_1$, $b_2$ 
associated to $s_1$ and $s_2$, respectively, 
lies in the image of the differential $d_2$.
\end{enumerate}
Then the derived group of $\Gamma$ is perfect.
\end{thm}

\begin{proof}
Let $ e\in T$ be an edge distinct from $e_0$,
and let $\gamma$ be a geodesic path from one of the endpoints of $e_0$
to one of the endpoints of $e$;
let $e_1$, \ldots, $e_k$ be the edges making up $\gamma$.
We may assume that $e_0 \neq e_1$ and $e_k \neq e$
and shall prove by induction on $k$ that the difference of the basis vectors corresponding to the end points of $e$ lies in the image of $d_2$.
If $k = 0$, then $e_0$ and $e$ have an endpoint, say $s_*$ in common.
Let $s$ be the endpoint of $e_0$ distinct from $s_*$; 
similarly, let $s'$ be the endpoint of $e$ distinct from $s_*$
and let $b$, $b_*$ and $b'$ denote the corresponding basis vectors.
The distance $d(s, s')$ is 2,
and so the generators $s$ and $s'$ commute by hypothesis (ii).
By formula \eqref{eq:Relator-r2-rank-1},
the image of $d_2$ contains therefore the vector
\[
(1-s) (b - b') = (1-s) (b-b_*) + (1-s) (b_* - b').
\]
Now, by hypothesis,  
the image of $d_2$ contains the vector $u = b_1 - b_2$
and this vector coincides, up to a sign, with $b-b_*$,
whence  $(1-s) (b_* - b') \in \im d_2$.
Hypothesis (i) implies next 
that the group $\Gamma$ satisfies the relation $(s' s_*)^m s' = s_* (s' s_*)^m$
for some positive integer $m$
and so it follows, 
as in section \ref{sssec:More-Artin-systems-with-4-generators},
that $b_* -b' \in \im d_2$.
If $k > 0$, the induction hypothesis guarantees 
that the difference $b_{0, e_k} - b_{1, e_k}$ lies in $\im d_2$, 
where $s_{0, e_k}$ and $s_{1, e_k}$ are the end points of the edge $e_k$,
and so it follows as before that $b_* -b' \in \im d_2$.
\end{proof}

There is one kind of Artin system  $(\Gamma, \Delta(S, E, \lambda))$ 
to which the above proposition applies 
and which will occur repeatedly in the sequel:
by section \ref{sssec:Artin-system-of-type-A4},
hypothesis (iii) holds, in particular, if the system $(\Gamma, \Delta(S, E, \lambda))$ 
contains a subsystem of type $A_4$.
Theorem \ref{thm:Artin-systems-with-perfect-derived-groups}
has therefore the following
\begin{crl}
\label{crl:Artin-systems-with-perfect-derived-groups}
Let  $(\Gamma, \Delta(S, E, \lambda))$ be an Artin system  
and assume the following hypotheses are satisfied:
\begin{enumerate}[(i)]
\item the graph $(S,E)$ is connected and admits a spanning tree $T$ 
all whose edges are labelled by odd numbers;
\item if $s$, $s'$ are vertices of $T$ with distance 2,
then the commuting relation $s s' = s' s$ holds in $\Gamma$;
\item there exist 4 generators in $S$, 
say $s_1'$, $s'_2$, $s'_3$ and $s'_4$,
which satisfy the 6 standard relations of an Artin system of type $A_4$
and have the property that the three edges $\{s'_1, s'_2\}$, $\{s'_2, s'_3\}$ and $\{s'_3, s'_4\}$ form a segment contained in $T$.
\end{enumerate}
Then the derived group of $\Gamma$ is perfect.
\end{crl}

\begin{proof}
The calculations in section \ref{sssec:Artin-system-of-type-A4} show, 
in particular, 
that the image of the differential $d_2$ contains all the differences 
$s'_{k}- s'_{k+1}$ for $k \in \{1,2,3\}$.
Hypothesis (iii) in Theorem \ref{thm:Artin-systems-with-perfect-derived-groups} 
is thus fulfilled and so the claim follows from that theorem. 
\end{proof}
%
%
\subsubsection{Return to Artin systems of type $A$}
\label{sssec:Artin-systems-of-type-A-return}
%
The relations imposed on an Artin systems 
$(\Gamma, \Delta(S, E, \lambda))$ of type $A_\ell$ 
with $\ell \geq 4$ guarantee 
that hypotheses (i), (ii) and (iii) of the preceding corollary are satisfied.
This corollary and the calculations at the beginning of the present section 
imply therefore
\begin{prp}
\label{prp:Derived-group-of-Artin-group-of-type-A}
Let $(\Gamma, \Delta(S, E, \lambda))$ be an Artin system of type $A_\ell$ 
and let $\Gamma'$ be the derived group of $\Gamma$.
For $\ell = 2$ the group $\Gamma'$ is then a free group of rank 2  
and it is perfect for  $\ell = 1$ and for $\ell \geq 4$.
If $\ell = 3$ then $\Gamma'_{\ab}$ is free abelian of rank 2.
\end{prp}
%
\subsection{Artin systems of types $D$ or $E$}
\label{ssec:Artin-systems-of-types-D-or-E}
%
Let $(\Gamma, \Delta(S, E, \lambda))$ be an Artin system of type $D_\ell$  
with $\ell \geq 4$.
If $\ell = 4$,  
the edges with labels 2 or 3 of the graph of $\Gamma$ form a triangular pyramid,
with $s_1$, $s_2$, $s_3$, say,  the vertices of the base and with apex $s_4$.
The group has the three commuting relations
\begin{equation}
\label{eq:Commuting-relations-D4}
s_1 s_2 = s_2 s_1, \quad
s_1 s_3 = s_3 s_1, \quad
s_2 s_3 = s_3 s_2
\end{equation}
and the three relations
\begin{equation}
\label{eq:Other-relations-D4}
s_1 s_4 s_1 = s_4 s_1 s_4, \quad
s_2 s_4 s_2 = s_4 s_2 s_4, \quad
s_3 s_4 s_3= s_4 s_3s_4.
\end{equation}
Set $Q =\Gamma_{\ab}$ and  $R = \Z{Q}$. 
The generators $s_k$ are mapped to the same element, 
say $s$, in $Q$;
so the differential $d_1$ is given by multiplication by the column vector 
$(1-s, 1-s, 1-s, 1-s)$.
If one denotes the standard basis vectors of $R^4$ by $b_1$,  \ldots, $b_4$,
the kernel of $d_1$ can be described thus:
\begin{equation}
\label{eq:ker-d1-type-D4}
\ker d_1 
= 
R \cdot (b_1- b_2)  \oplus R \cdot (b_1- b_3)  \oplus R \cdot (b_1- b_4). 
\end{equation}
So $\ker d_1$ is the free $R$-module with basis 
\[
u = b_1-b_2, \quad v = b_1 - b_3, \quad \text{and}\quad w = b_1 - b_4.
\]
The differential $d_2$ is given by a matrix $M(D_4)$
with 6 rows, 4 columns and entries in $R$.
By equation \eqref{eq:Relator-r2-rank-1},
the rows corresponding to the relations listed in
\eqref{eq:Commuting-relations-D4},
are
\begin{equation}
\label{eq:Rows-L1-L2-L3-of-M(D4)}
L_1 =  (1-s) \cdot u,\quad
L_2 =  (1-s)  \cdot v, \quad 
L_3 = (1-s) \cdot (v - u).
\end{equation}
By equation \eqref{eq:Row-for-r3}
the rows corresponding to the relations listed in
\eqref{eq:Other-relations-D4}
are
\begin{align}
L_4 &= (1-s + s^2) \cdot (b_1 - b_4) = (1-s+s^2) \cdot w,
\label{eq:Row-L4-of-M(D4)}\\
L_5 &= (1-s + s^2) \cdot (b_2 - b_4) = (1-s+s^2)\cdot (w -u),
\label{eq:Row-L5-of-M(D4)}\\ 
L_6 &= (1-s + s^2) \cdot (b_3 - b_4) = (1-s + s^2) \cdot (w- v).
\label{eq:Row-L6-of-M(D4)}
\end{align}
Since $L_3= L_2 -L_1$, the row $L_3$ is superfluous.
Next $L_4 - L_5 = (1-s+s^2) u$ and so $L_4 - L_5 + s \cdot L_1 = u$.
Similarly,  $L_4 - L_6 + s \cdot L_2 = v$.
It follows that the homology group $\ker d_1/ \im d_2$ is isomorphic to
\[
\left(R u \oplus R v \oplus R w\right) / 
\left(R u \oplus R v \oplus R (1 - s + s^2) w\right)
\]
and so $\Gamma'_{\ab}$ is a free abelian group of rank 2. 

Assume next that $(\Gamma, \Delta(S, E, \lambda))$ 
is an Artin system of type $D_\ell$ with $\ell > 4$.
The system contains then a subsystem of type $A_4$,
and it follows, as in section \ref{sssec:Artin-systems-of-type-A-return},
that $\Gamma'$ is perfect.
 
 Suppose, finally, 
 that $(\Gamma, \Delta(S, E, \lambda))$ is an Artin system of type $E_\ell$
 with $\ell$ in $\{6,7,8\}$.
The restriction on $\ell$ is imposed for the following reasons:
if one omits one of the terminal edges of a Coxeter graph of type $E_6$ 
one obtains a Coxeter graph of type $A_5$ or of type $D_5$.
If, on the other hand, 
one adds an edge to the longest chain in a Coxeter graph of type $E_8$
one obtains a Coxeter graph 
whose corresponding Coxeter group is no longer finite
(see, \eg \cite[Section 2.5]{Hum90}).

Suppose now
that  $\Delta$ is a Coxeter graph of one of the types 
$E_6$, $E_7$ or $E_8$.
Then its Artin system contains a subsystem of type $A_4$ 
and so it follows as before 
that  the derived group of $\Gamma$ is perfect.

\begin{remark}
\label{remark:Generalised-groups-of-of-type-E}
The Coxeter graph $\Delta$ of an Artin system of type $E$ 
has the form of a tripod.
If each leg of this tripod has length 1, 
the graph is that of an Artin system of type $D_4$.
If each leg has positive length 
and at least one of them has length greater than 1,
and if the generators in the tripod with distances 2 or 3 commute,
the Artin system contains a subsystem of type $A_4$
whence the derived group of the Artin group in question is perfect 
by Corollary \ref{crl:Artin-systems-with-perfect-derived-groups}.
Examples of Coxeter graphs with the stated properties can be found
on page 144 in \cite{Hum90}.
\end{remark}

Here is a summary of the results obtained in this section
\ref{ssec:Artin-systems-of-types-D-or-E}:
\begin{prp}
\label{prp:Derived-group-of-Artin-group-of-types-D-and-E}
Assume first 
that $(\Gamma, \Delta(S, E, \lambda))$ is an Artin system of type $D_\ell$ 
with $\ell \geq 4$
and consider the derived group  $\Gamma'$ of $\Gamma$.
If $\ell = 4$ then $\Gamma'_{\ab}$ is free abelian of rank 2
and if $\ell > 4$ then $\Gamma'$ is perfect.
If, secondly,
$(\Gamma, \Delta(S, E, \lambda))$ is an Artin system of type $E_6$, $E_7$ or $E_8$
 its derived group is perfect.
\end{prp}

%
\subsection{Artin systems of type $H_3$ or $H_4$}
\label{ssec:Artin-systems-of-finite-type-H3-or-H4}
%
I continue with the two Artin systems of finite type 
whose Coxeter graphs contain an edge with label 5. 
%
\subsubsection{Artin system of type $H_3$}
\label{sssec:Artin-system-of-type-H3}
%
%
Let  $(\Gamma, \Delta(S, E, \lambda))$ be an Artin system of type $H_3$ 
with standard generating set $S = \{s_1, s_2, s_3 \}$.
The analysis of $\Gamma$ will begin with a first part 
that is similar to beginning of the analysis of an Artin group of type $A_3$,
given in section \ref{sssec:Artin-system-of-type-A3}.
The standard defining relators of $(\Gamma, S)$ are
\begin{equation}
\label{eq:Relators-H3}
r_2 = s_1 s_3 \cdot (s_3 s_1)^{-1},
\quad
r_3 = s_1 s_2 s_1 \cdot  (s_2 s_1s_2)^{-1},
\quad
r_5 = s_2 s_3 s_2 s_3 s_2\cdot  (s_3 s_2s_3s_2s_3)^{-1}.
\end{equation}
The abelianisation of $\Gamma$ is infinite cyclic, 
generated by the common image,
say $s$, of $s_1$, $s_2$ and $s_3$ in $Q = \Gamma_{\ab}$.
Set $R = \Z{Q}$ to simplify notation.

The kernel of $d_1$ is a free submodule of 
$R b_1 \oplus R  b_2 \oplus R  b_3$,
namely
\begin{equation}
\label{eq:Kernel-d1-type-H3}
\ker d_1 = R \cdot (b_1 - b_2) \oplus  R \cdot (b_2 - b_3)
\end{equation}
(see formula \eqref{eq:kernel-d1-rank-1}).
Set $u = b_1 - b_2$ and $v = b_2-b_3$.
In view of equations 
\eqref{eq:Relator-r2-rank-1}, \eqref{eq:Row-for-r3} and \eqref{eq:Row-for-r5},
the image of the differential $d_2$ is generated by the three vectors
\[
x = (1-s )(e_1 - e_3) = (1-s) (u + v),
\quad
y=  (1 - s + s^2) u \text{ and }
z= (1 - s + s^2 - s^3 + s^4) v.
\]
These three vectors generate all of $\ker d_1$.
Indeed:
\begin{align*}
L_1 &= s x + y = u + (s-s^2) v,\\
L_2 &=(1+s^2) L_1 + z - y
=
\left( ( 1 + s^2) u - y \right) + \left((s -s^2 + s^3 - s^4) v + z \right)\\
&= su + v,\\
L_3 &= (1 - s + s^2) L_2 - sy 
= \left((s - s^2 + s^3) u  - s  y \right) + (1-s + s^2) v, \\
&= (1- s + s^2) v,\\
L_4 &= z - s^2 L_3 = (1-s) v, \text{ whence}\\
L_5 &= z +(s + s^3) L_4 = v \quad \text{ and } \quad
L_6 = L_1 - (s - s^2) L_5 = u.
\end{align*}
%
\subsubsection{Artin system of type $H_4$}
\label{sssec:Artin-system-of-type-H4}
%
Let  $(\Gamma, \Delta(S, E, \lambda))$ be an Artin system of type $H_4$
with generating set $\{s_1, s_2, s_3, s_4\}$. 
Since the label of the edge $\{s_3, s_4\}$ is 5, 
the system contains a subsystem of type $H_3$ with. generating set 
$S_1 = \{s_2, s_3, s_4\}$
The preceding section  \ref{sssec:Artin-system-of-type-H3}
then shows that $\ker d_2$ contains the differences $b_2- b_3$ and $b_3 - b_4$
and so $\Gamma'$ is perfect by Theorem
\ref{thm:Artin-systems-with-perfect-derived-groups}.
\smallskip

If one reflects about the proof just given 
one sees that it can be generalized 
so as to establish a companion result of Corollary
\ref{crl:Artin-systems-with-perfect-derived-groups},
namely
\begin{crl}
\label{crl:Artin-systems-with-perfect-derived-groups-II}
Let  $(\Gamma, \Delta(S, E, \lambda))$ be an Artin system  
for which the following hypotheses are satisfied:
\begin{enumerate}[(i)]
\item the graph $(S,E)$ is connected and admits a spanning tree $T$ 
all whose edges are labelled by odd numbers;
\item if $s$, $s'$ are vertices of $T$ with distance 2,
then the commuting relation $s s' = s' s$ holds in $\Gamma$;
\item there exist 3 generators, 
say $s_1'$, $s'_2$ and $s'_3$, in $S$
which satisfy the 3 standard relations of an Artin system of type $H_3$
and have the property 
that $\{s'_1, s'_2\}$ and $\{s'_2, s'_3\}$ are adjacent edges of $T$.
 \end{enumerate}
Then the derived group of $\Gamma$ is perfect.
\end{crl}
\begin{examples}
\label{examples:crl:Artin-systems-with-perfect-derived-groups-II}
Here are some examples of Artin systems
that arise in the study of Coxeter groups 
and which illustrate Corollary
\ref{crl:Artin-systems-with-perfect-derived-groups-II}.
\begin{enumerate}[(i)]
\item  In the analysis of positive semidefinite graphs,
as carried out in \cite[Section 2.5]{Hum90}, 
the author considers graphs called $Z_4$ and $Z_5$ 
and displayed on page 35.
The Artin systems associated to these graphs satisfy 
the assumptions of the corollary.
\item In \cite{Hum90} the author lists on pages 141-- 141  graphs of the hyperbolic Coxeter groups.
Many of the associated Artin systems satisfy the assumptions
of either Corollary \ref{crl:Artin-systems-with-perfect-derived-groups}
or of Corollary \ref{crl:Artin-systems-with-perfect-derived-groups-II}.
\end{enumerate}
\end{examples}

%
\subsection{Artin system of type $F_4$}
\label{ssec:Artin-system-of-type-F4}
%
There are three kinds of irreducible Artin systems of finite type 
whose abelianisations are free abelian of rank 2,
 the system of type $F_4$, 
 those of type $B$ and the systems of type $I_2(2n)$ with $n \geq 3$. 
 The systems of type $I_2(2n)$ have been investigated in section
 \ref{ssec:Two-generator-Artin-systems}.
 In this section that of type $F_4$ will be studied;
systems of type $B$ will be the topic of section 
 \ref{ssec:Artin-systems-of-type-B}. 
 
Let $(\Gamma, \Delta(S, E, \lambda))$ be an Artin system of type $F_4$ 
with $S = \{s_1, s_2, s_3, s_4\}$. 
The group has 6 defining relations,
the three commuting relations
\begin{equation}
\label{eq:Commuting-relations-F4}
s_1 s_4 = s_4 s_1, \quad
s_1 s_3 = s_3 s_1, \quad
s_2 s_4 = s_4 s_2.
\end{equation}
and the relations
\begin{equation}
\label{eq:Other-relations-F4}
s_1 s_2 s_1 = s_2 s_1 s_2, \quad
(s_2 s_3)^2 = (s_3 s_2)^2, \quad
s_3 s_4 s_3 = s_4 s_3 s_4.
\end{equation}
Note that 5 out of the 6  relations listed in equations
\eqref{eq:Commuting-relations-F4} and \eqref{eq:Other-relations-F4}
are identical with those of the Artin system of type $A_4$,
treated in section  \ref{sssec:Artin-system-of-type-A4}.
The fifth relations, however, differ significantly:
in the case of the Artin system of type $F_4$ 
it is responsible for the fact 
that $\Gamma_{\ab}$ is not infinite cyclic, 
but free abelian of rank 2.

To compute the abelianisation of $\Gamma'$,
we determine, as in the preceding sections, the homology group 
$A = H_1(\Gamma, \Z{\Gamma_{\ab}})$;
it can be computed 
with the beginning of the $\Z{\Gamma}$-free resolution of $\Z$ 
associated to the standard presentation of the Artin system $(\Gamma, S)$,
and it is the quotient of the kernel of 
$d_1$
modulo the image of 
\[
d_2 \colon  \bigoplus\nolimits_{1  \leq i \leq 6}\Z{\Gamma_{\ab}} 
\to \bigoplus\nolimits_{1  \leq j \leq 4}\Z{\Gamma_{\ab}}.
\]
Let $\ab \colon \Gamma  \epi \Gamma_{\ab}$ 
denote the canonical map onto the abelianisation of $\Gamma$. 
Under $\ab$ 
the generators $s_1$ and $s_2$ map to the same element, 
say $s$, 
and $s_3$, $s_4$ go to a second element, say $t$.
These two elements $s$ and $t$ form a basis of $\Gamma_{\ab}$.
To simplify notation, 
set $Q = \Gamma_{\ab}$ and $R = \Z{Q}$.
The differential $d_1$  is given by multiplication by the column vector 
$(1-s, 1-s, 1-t, 1-t)$ and the kernel is
\begin{equation}
\label{eq:kernel-d1-F4}
\ker d_1 = R  (b_1 - b_2) \oplus R \cdot \left((1-t) b_2 + (s-1) b_3 \right)
\oplus R (b_3-b_4).
\end{equation}
A short calculation then shows 
that $\ker d_1$ is the free $R$-module with basis
\begin{equation}
\label{eq:basis-ker-d1-type-F4} 
u = b_1 - b_2, \quad v = (1-t)b_2 +(s-1)b_3, \quad w = b_3 -b_4.
\end{equation}

The differential $d_2$ is given by a matrix $M(F_4)$ with entries in $R$;
it has 6 rows, one row for each defining relator, and 4 columns.
By equation \eqref{eq:Relator-r2-rank-2}
the rows corresponding to the three commuting relations listed in
\eqref{eq:Commuting-relations-F4}
are
\begin{align*}
\label{eq:Rows-L1-L2-L3-of-M(H4)}
L_1 &= (1-t) b_1 + (s-1) b_4 =(1-t) (b_1 -b_2) + v + (1-s) (b_3 - b_4) 
\\
&= (1-t)u + v + (1-s) w,\\
L_2 &=  (1-t) b_1 + (s-1) b_3  = (1-t) u + v,\\
L_3 &=  (1-t) b_2 + (s-1) b_4  =  v + (1-s) w.
\end{align*}
Set $L_0 = L_2 + L_3 - L_1$.
Then
\begin{equation}
\label{eq:New-basis}
L_0= v,
\quad
L_2 - L_0 = (1-t) u
\quad\text{and} \quad
L_3 -L_0 = (s-1) w.
\end{equation}
Note that the three elements displayed in equation \eqref{eq:New-basis}
generate the same $R$-submodule 
as do $L_1$, $L_2$ and $L_3$.

Now to the three relations listed in equation 
\eqref{eq:Other-relations-F4}; 
they give rise to the last three rows of the matrix $M(F_4)$.
The first and the third relations
 coincide with those for the Artin group of type $A_4$; 
they are
\begin{equation}
\label{eq:Rows-L4-and-L6-type-F4}
L_4 = (1-s + s^2) u
\quad\text{and}\quad
L_6 = (1-t+t^2) w.
\end{equation}
The second relation listed in \eqref{eq:Other-relations-F4}
yields the row
\[
(1+st)\left((1-t) b_2 +(s-1) b_3\right) = (1+st) v
\]
(see equation \eqref{eq:Row-for-r4});
it is redundant in view of the first entry in equation \eqref{eq:New-basis}.
It follows that the homology group $\ker d_1/\im d_2$ 
is isomorphic to the direct sum 
\[
R / (R (1 - s + s^2 ) + R ( 1-t))  
\;\oplus\;
R/R
\; \oplus \;
R / (R (1 - s ) + R ( 1-t + t^2)).
\]
A short calculation then discloses 
that the abelian groups underlying each of the non-trivial summands in the above sum are free abelian of rank 2.
This proves:
\begin{lem}
\label{lem:Derived-group-of-Artin-group-of-type-F4}
If $(\Gamma, \Delta(S, E, \lambda))$ is an Artin system of type $F_4$
then $\Gamma'_{\ab}$ is free abelian of rank 4
and so $\Gamma'$ is not perfect.
\end{lem}
\begin{remark}
According to  \cite[Section 3.3.6]{MuRo06} 
the derived group $\Gamma'$ of $\Gamma$ is finitely generated.
 \end{remark}
%
\subsection{Artin systems of type $B$}
\label{ssec:Artin-systems-of-type-B}
%
The study of Artin systems 
of type $B$ 
splits into several cases.
 The systems of types $B_2$, $B_3$ and $B_4$ will be handled separately,
 those of the remaining types are investigated by an inductive argument
 based on the result for the Artin system of type $B_5$.
 %
\subsubsection{Artin system of type $B_2$}
\label{sssec:Artin-system-of-type-B2}
%
Let $(\Gamma, \Delta(S, E, \lambda))$ be an Artin system of type $B_2$;
this system is identical with that of type $I_2(4)$.
Thus Proposition \ref{prp:Two-generator-Artin-groups-with-even-label} applies
and shows that $\Gamma'_{\ab}$ is free abelian of infinite rank.
%
\subsubsection{Artin system of type $B_3$}
\label{sssec:Artin-system-of-type-B3}
%
Let $(\Gamma, \Delta(S, E, \lambda))$ be an Artin system of type $B_3$.
The standard generators of $\Gamma$ are $s_1$, $s_2$, $s_3$ 
and the defining relations are these:
\begin{equation}
\label{eq:Defining-relations-of-B3}
s_1 s_3 = s_3 s_1, \quad s_1 s_2 s_1 = s_2 s_1 s_3 
\quad \text{and}\quad
(s_2 s_3 )^2 = (s_3 s_2)^2.
\end{equation}
In view of  section \ref{sssec:d2-and-im-d2}
these relations are equivalent to the relators
\begin{equation}
\label{eq:Relators-B3}
r_2 = s_1 s_3 \cdot s_1^{-1} s_3^{-1},
\quad
r_3 = s_1 s_2 s_1 \cdot  s_2^{-1} s_1^{-1}s_2^{-1},
\quad
r_4 = (s_2 s_3)^2\cdot  (s_2^{-1} s_3^{-1})^2.
\end{equation}
Set $Q = \Gamma_{\ab}$ and $R = \Z{Q}$.
The images of the generators $s_1$, $s_2$ coincide in $Q$,
call their common image $s$,
and $s_3$ maps to an element $t \in Q$,
all in such a way that $\{s, t\}$ is a basis of the free abelian group $Q$.
The kernel of the homomorphism $d_1$ is the submodule 
\begin{equation}
\label{eq:kernel-d1-type-B3}
\ker d_1 = R \cdot (b_1 - b_2) \oplus  R \cdot ((1-t)b_2 + (s-1)b_3)
\end{equation}
of the $R$-free module $R b_1 \oplus R  b_2 \oplus R  b_3$
(see formula \eqref{eq:kernel-d1-type-B}).

Set $u = b_1 - b_2$ and $v = (1-t)b_2 + (s-1)b_3$.
In view of equations
\eqref{eq:Relator-r2-rank-2}, 
\eqref{eq:Row-for-r3} and \eqref{eq:Row-for-r4},
the image of the differential $d_2$ is then generated by the three vectors
\[
L_1 = (1-t)b_1 + (s-1)b_3 = (1-t)u +  v.
\quad 
L_2 =(1 - s + s^2)\cdot u,
\quad
L_3= (1+st) \cdot v.
\]
The kernel of $d_1$ is the free $R$-module with basis $\{u, v\}$,
hence also the free $R$-module with basis $\{u, L_1 = (1-t)u + v \}$,
and  $L_3 = (1+ st) v =  (1+ st) L_1 - (1 + st) (1-t))u$.
The homology module $\ker d_1 / \im d_2$ is therefore  isomorphic 
to the module 
\begin{equation}
\label{eq:Homology-for-B3}
R/(I + J)
\quad \text{where} \quad I = R \cdot (1 - s + s^2) \text{ and } 
J = R \cdot(1+st) (1-t).
\end{equation}

We are left with the determination of the abelian group 
underlying the ring $R/(I+J)$.
It will be carried out in two steps. 
To begin with,
set
\[
\bar{S} = \Z[s, s^{-1}]/  \left(\Z[s, s^{-1}] \cdot (1 - s + s^2) \right).
\]
Since the constant term of the polynomial $ f(s) = 1 - s + s^2$ is a unit in $\Z$,
this ring is isomorphic to $\Z[s]/(\Z[s] \cdot f(s))$
 and hence a ring of integers contained in the algebraic number field 
 $\Q[s]/(1 - s + s^2)$.
 The additive group of $\Z[s]/(\Z[s] \cdot f(s))$ is free abelian of rank 2.
In the second step; 
the ring $R/(I + J)$ will be reexpressed with the help of $\bar{S}$.
Let $\bar{s}$ denote the canonical image of $s$ in $\bar{S}$ and set
\[
p(t) =(1 + \bar{s}t) (t-1) = -1 + (1 - \bar{s})t + \bar{s}t^2 \in \bar{S}[t, t^{-1}].
\]
Then $p(t)$ is a quadratic polynomial with coefficients in $\bar{S}$ 
whose constant and leading terms are units in $\bar{S}$.
In addition, the ring $R/(I + J)$ is isomorphic to the ring
\[
T = \bar{S}[t, t^{-1}] / \bar{S}[t, t^{-1}] \cdot p(t).
\]
The following elementary result allows one therefore to conclude 
that the ring $T$, viewed a $\bar{S}$-module, is a free module of rank 2. 

\begin{lem}
\label{lem:T-as-bar-S-module}
Let $S$ be commutative ring (with unit $ 1 \neq 0$) 
and let $p(X)$ be a Laurent-polynomial in the ring $S[X, X^{-1}]$.
Suppose $p(X) = p_0 + p_1 X +\cdots + p_m X^m$
where $p_0$ as well as $p_m$ are units in $S$,
and set
\[
T =S[X, X^{-1}] / \left(S[X, X^{-1}] \cdot p(X)\right).
\]
Then $T$, viewed as an $S$-module, is a free module of rank $m$.
\end{lem}
\begin{proof}
Let $J$ denote the principal ideal of $S[X, X^{-1}]$ generated by $p(X)$.
Since the trailing coefficient $p_0$ of $p(X)$ is a unit of $S$, 
every element $f \in S[X, X^{-1}]$ is congruent \emph{modulo} $J$ to a polynomial $g \in S[X]$. 
Since the leading coefficient of $p(X)$ is a unit in $S$, 
the polynomial $g$ is congruent \emph{modulo} $J$ 
to a polynomial of degree at most $m-1$.
As an $S$-module, 
the ring $T$ 
is thus generated by the canonical images of the monomials 1, $X$, \ldots, $X^{m-1}$. 
We are left with proving that these monomials are $S$-linearly independent.
This amounts to show 
that the zero-polynomial  is the only polynomial of degree at most $m$
which lies in the ideal $S[X, X^{-1}] \cdot p(X)$.
Let $U$ be the $S$-submodule of $S[X,X^{-1}]$ generated 
by the elements $1$, $X$, \ldots, $X^{m-1}$.

Suppose next that $f$ is an element of $S[X,X^{-1}]$ 
so that the product $f \cdot p(X)$ lies in $S[X]$.
Since the trailing coefficient of $p(X)$ is a unit, 
hence not a zero divisor, 
the element $f$ must be a polynomial.
If $f$ is the zero polynomial, all is well; 
otherwise, 
the product $f \cdot p(X)$ 
is a polynomial of degree at least $m$ and thus outside of $U$.
\end{proof}
 %
%
\subsubsection{Artin system of type $B_4$}
\label{sssec:Artin-system-of-type-B4}
%
Let $(\Gamma, \Delta(S, E, \lambda))$ be an Artin system of type $B_4$.
The standard generators of $\Gamma$ are $s_1$, $s_2$, $s_3$, $s_4$, 
and the standard defining relations are, firstly, 
the 3 commuting relations
\begin{equation}
\label{eq:Commuting-relations-of-B4}
s_1 s_4 = s_4 s_1, \quad 
s_1 s_3 = s_3 s_1 
\quad \text{and}\quad
s_2 s_4 = s_4 s_2
\end{equation}
and then the 3 additional relations
\begin{equation}
\label{eq:Additional-relations-of-B4}
s_1 s_2 s_1 = s_2 s_1 s_2,
\quad
s_2 s_3 s_2 = s_3 s_2 s_3,
\quad \text{and}\quad
(s_3 s_4 )^2 = (s_4 s_3)^2.
\end{equation}
These 6 relations are equivalent to the 6 relators
\begin{gather}
\label{eq:Relators-B4}
s_1 s_4 \cdot s_1^{-1} s_4^{-1},
\quad
s_1 s_3 \cdot s_1^{-1} s_3^{-1},
\quad
s_2 s_4 \cdot s_2^{-1} s_4^{-1}, \\
r_3 =s_1 s_2 s_1 s_2^{-1} s_1^{-1}s_2^{-1},
\quad
r_3' =s_2 s_3 s_2   s_3^{-1} s_2^{-1}s_3^{-1},
\quad
r_4 = (s_3 s_4)^2 (s_3^{-1} s_4^{-1})^2.
\end{gather}

Set $Q = \Gamma_{\ab}$ and $R = \Z{Q}$.
The generators $s_1$, $s_2$ and $s_3$ have the same image in $Q$,
call it $s$,
and $s_4$ maps to an element $t \in Q$,
all in such a way that $s$ and $t$ are free generators of $Q$.
The kernel of the homomorphism $d_1$ is the submodule 
\begin{equation}
\label{eq:kernel-d1-type-B4}
\ker d_1 = R \cdot (b_1 - b_2) 
\oplus  
R \cdot (b_2 - b_3) \oplus R \cdot ((1-t)b_3 + (s-1)b_4)
\end{equation}
of the free $R$-module $R b_1 \oplus R  b_2 \oplus R  b_3 \oplus R  b_4$
(see formula \eqref{eq:kernel-d1-type-B}).
Set $u = b_1 - b_2$, set $v = b_2 - b_3$  and
$w= (1-t)b_3 + (s-1)b_4$;
these elements form a basis of $\ker d_1$.

Now to the image of $d_2$; 
it is generated by 6 elements, 
one element for each defining relator.
In view of equations
\eqref{eq:Relator-r2-rank-2}, 
\eqref{eq:Row-for-r3} and \eqref{eq:Row-for-r4},
these 6 elements are
\begin{align*}
L_1 &= (1-t)b_1  + (s-1)b_4 = (1-t)u + (1-t)v + w,\\
L_2 &= (1-s)(b_1 - b_3) = (1-s) (u + v), \\
L_3 &= (1-t)b_2  + (s-1)b_4 = (1-t)v + w,\\
L_4 &= (1 - s + s^2) \cdot u, \\
L_5 &= (1 - s + s^2) \cdot v, \\
L_6&= (1 + st) \cdot  w.
\end{align*}
It follows, first of all, that image of $d_2$ contains the vectors
\[
L_7 = s \cdot L_2 + L_4 + L_5 = u + v 
\quad \text{and} \quad
L_8 = L_1 - L_3 = (1-t) u.
\] 
The difference $(1-t) (u+v) - L_8$ equals $L_9 = (1-t)v$, 
and $L_3 - L_9 = w$.
So $\im d_2$ contains the vectors $u+v$ and $w$.
The image of $d_2$ is therefore also generated by the vectors
\[
w, \quad u + v, \quad (1-s + s^2) \cdot  u \quad \text{and} \quad (1-t) \cdot u.
\]
Since $\ker d_1$ is the free $R$-module with basis  $u$, $v$ and $w$,
the $R$-module $\ker d_1/ \im d_2$ is therefore isomorphic with
\begin{equation}
\label{eq:Description-of-homology-module-for-B4}
R /\left( R (1 -s + s^2) + R(1-t) \right).
\end{equation}
Now $R$ is the group ring of the free abelian group with basis $\{s, t\}$.
The preceding description of $\ker d_1/ im d_2$ allows one therefore to conclude 
that the abelian group underlying $\ker d_1/ \im d_2$  is free abelian of rank 2.
%
\subsubsection{Artin systems of type $B_\ell$ with $\ell \geq 5$}
\label{sssec:Artin-systems-of-type-Bell}
Let $(\Gamma, \Delta(S, E, \lambda))$ be an Artin system of type $B_\ell$ 
with standard generating set  $S = \{s_1, \ldots,  s_{\ell-1} , s_\ell \}$.
Set $Q = \Gamma_{\ab}$ and $R = \Z{Q}$, 
and assume that $\ell \geq 5$.
The generators $s_1$, \ldots, $s_{\ell-1}$ 
map onto one and the same element, say $s$, in  $Q$; 
the generator $s_\ell$ maps onto an element, say $t$, in $Q$,
and $s$ and $t$ are free generators of the free abelian group $Q$.
The differential $d_1$ is thus given by multiplication by the column vector 
$\left(1-s, \ldots, 1-s, 1-t \right)$
and
\begin{equation}
\label{eq:kernel-d1-type-B-bis}
\ker d_1 
= 
R \cdot (b_1 - b_2)  \oplus \cdots 
\oplus 
R\cdot (b_{\ell-2}-b_{\ell-1}) 
\oplus 
R\cdot ((1-t)b_{\ell-1} + (s-1)b_\ell)
\end{equation}
by formula\eqref{eq:kernel-d1-type-B}).
The differential $d_2$ is given by a matrix with $\tfrac{1}{2}(\ell-1)\ell$ rows
and $\ell$ columns.
The subset $S_1 = \{s_1, \ldots, s_{\ell-1}\}$ of $S$
generates a subgroup $\Gamma_1$ of $\Gamma$
and $(\Gamma_1, \Delta_1(S_1, E_1, \lambda_1))$ 
is an Artin system of type $A_{\ell- 1}$.
Because $\ell - 1 \geq 4$,
Proposition \ref{prp:Derived-group-of-Artin-group-of-type-A}
applies and shows
that $\im d_2$ contains the vectors $b_{k} - b_{k+1}$ 
for $k \in \{1, 2,  \ldots, \ell- 2\}$;
so $\im d_2$  contains, in particular, the vector $b_{\ell-2} - b_{\ell-1}$.
Moreover, since the relation 
$s_{\ell-2} s_{\ell} =s_{\ell} s_{\ell-2}$ is a standard relation of $\Gamma$,
formula \eqref{eq:Relator-r2-rank-2} guarantees
that the vector $(1-t) b_{\ell-2} + (s-1) b_\ell$ lies in the image of $d_2$.
But if so, the vector 
\[
(1-t) b_{\ell-2} + (s-1) b_\ell + (1-t) (b_{\ell-1} - b_{\ell-2})
=
(1-t) b_{\ell-1} + (s-1) b_\ell 
\]
is an element of $\im d_2$.
As the differences $b_2-b_1$, \ldots, $b_{\ell-1}-b_{\ell-2}$ 
are already known to belong to $\im d_2$
it follows from equation \eqref{eq:kernel-d1-type-B-bis}
that $\im d_2 = \ker d_2$ or, in other words, that $\Gamma'$ is perfect.
\smallskip

The results obtained in this section \ref{ssec:Artin-systems-of-type-B} 
can be summarized as follows:
\begin{prp}
\label{prp:Derived-groups-of-Artin-groups-of-type-B}
Let $(\Gamma, \Delta(S, E, \lambda))$ be an Artin system of type $B_\ell$ 
with $\ell \geq 2$
and set $A = \Gamma'_{\ab}$. 
If $\ell = 2$, the group $A$ is free abelian of infinite rank, 
if $\ell$ is 3  or 4 it is free abelian of rank 4 or 2, respectively, 
and if $ \ell \geq 5$ it is trivial.
\end{prp}
%
\subsection{Artin systems $(\Gamma, S)$ of type $I_2(2n)$ with $n \geq 3$}
\label{ssec:Artin-systems-of-type-I2(even)}
If $\Gamma, S)$ is an Artin system with standard generators $\{s_1, s_2\}$
and an edge label that is a positive and even integer,
$\Gamma'_{ab}$ is free abelian of infinite rank 
by Proposition \ref{prp:Two-generator-Artin-groups-with-even-label}
and so $\Gamma'$ is not perfect.
%
%
\subsection{Concluding remarks}
\label{ssec:Concluding-remarks}
The main focus of the present Section \ref{sec:Artin-systems-of-finite-type}
 has been on the question
\emph{whether or not the derived group $\Gamma'$
of an irreducible Artin group $\Gamma$ of finite type is perfect}.
The answer has been found for each of these groups,
and, in many cases, 
more precise information about $\Gamma'$ could also be obtained.

Here is a list of the main insights about irreducible Artin systems of finite type:
\begin{enumerate}[(i)]
\item
The abelianisation of $\Gamma$ is either infinite cyclic
or free abelian of rank 2.
\item If $\Gamma_{\ab}$ is \emph{infinite cyclic}
its graph contains a spanning tree all whose edges have label 3 or 5,
and so $\Gamma'$ is finitely generated 
by Proposition \ref{prp:Artin-systems-with-infinite-cyclic-abelianisation}.
\item If $\Gamma_{\ab}$ is \emph{infinite cyclic}
the abelian group $A =\Gamma'_{\ab}$ is free abelian of even rank.
It is trivial if $\Gamma$ has one of the types
$A_1$ or  $A_\ell$ with $\ell \geq 4$,
or of type $D_\ell$ with $\ell \geq 5$,
or of one of the types $E_6$, $E_7$, $E_8$, $H_3$,  $H_4$;
the group $A$ has rank 4 if the type of $\Gamma$ is $B_3$ 
and it has rank 2 if $\Gamma$ has one of the types
$A_2$, $A_3$, $B_4$ or $D_4$.
Finally, the group $A$ is of rank $2n$ if $\Gamma$ is of type $I_2(2n+1)$ 
with $m \geq 2$.

\item If $\Gamma_{\ab}$ is \emph{free abelian of rank 2}  the Coxeter graph of $\Gamma$
is a line that contains an edge with label 4.
The system is then either of type $F_4$, 
or of type $B_\ell$ with $\ell \geq 2$, or of type $I_2(2n)$ with $n \geq 3$.
In the first case $A = \Gamma_{ab}$ is free abelian of rank 4;
if $\Gamma$ is of type $B_\ell$ 
then $A$  is free abelian of infinite rank for $\ell = 2$,
it is free abelian of rank 4 if $\ell = 3$ and free abelian of rank 2 if $\ell = 4$, 
and it is reduced to 0 for $\ell \geq 5$.
If the system is of type  $I_2(2n)$ with $n \geq 3$, 
the group $A$ is free abelian of infinite rank.
\end{enumerate}
%
\subsubsection{Comments on Theorem 
\ref{thm:Artin-systems-with-perfect-derived-groups}
and on its Corollaries \ref{crl:Artin-systems-with-perfect-derived-groups}
and \ref{crl:Artin-systems-with-perfect-derived-groups-II}}
\label{sssec:Comments-on-perfectness-of-derived-group}
%
Among the results obtained in Section \ref{sec:Artin-systems-of-finite-type}
there are some 
which deal with \emph{Artin systems of infinite type}.
Such result are Theorem 
\ref{thm:Artin-systems-with-perfect-derived-groups}
and its Corollaries
\ref{crl:Artin-systems-with-perfect-derived-groups}
and \ref{crl:Artin-systems-with-perfect-derived-groups-II}.
All three results list conditions 
which imply that the derived group of the considered Artin system 
$(\Gamma,\Delta(S, E, \lambda))$ is perfect.
The common assumptions are of two kinds:
hypothesis (i) requires 
that the graph of the system contain a spanning tree $T$
all whose edges are labelled by odd numbers,
while hypothesis (ii) demands
that the commuting relation $s s' = s's$ be satisfied 
for every pair $\{s,s\}$ of vertices in the spanning tree $T$ 
which are at distance 2. 
As we shall see below 
these two hypotheses do not imply 
that $\Gamma'$ is perfect.

Theorem 
\ref{thm:Artin-systems-with-perfect-derived-groups}
and its corollaries list therefore a third hypothesis.
In the theorem it requires 
that there exist an edge $e_0 = \{s_1, s_2\}$ in the spanning tree $T$ 
such that the image of $d_2$ contain the difference $b_1- b_2$ of the basis vectors associated to $s_1$ and $s_2$, respectively.
This condition holds 
if the Artin system contains a subsystem of type $A_4$,
a condition formulated in Corollary 
\ref{crl:Artin-systems-with-perfect-derived-groups}
and implying that the spanning tree $T$ contains two generators $s_1$, $s_4$ 
which are at distance 3 and commute.
The condition holds also 
if the system contains a subsystem of type $H_3$
(see Corollary \ref{crl:Artin-systems-with-perfect-derived-groups-II}).
If, however, $S$ has no pair of generators which commute and are at distance 3 
and if the labels of the spanning tree are all equal to the the same odd number greater than 1,
the derived group of $\Gamma$ need not be perfect,
as the following examples show.

\begin{example}
\label{examle-I2(odd)}
Let $(\Gamma, \Delta(S, E, \lambda))$ be an Artin system of type $I_2(2m + 1)$ 
and assume that $m > 0$.
Then $\Gamma'$ is a free group of rank $2m$
(see Proposition \ref{prp:Two-generator-Artin-groups-with-odd-label}).
\end{example}

\begin{example}
\label{example:Generalisation}
Assume $(\Gamma, \Delta(S, E, \lambda))$  is an Artin system 
whose graph has edges with only two labels,
the number 2 and an odd number $2m + 1$ greater than 1. 
Assume, in addition, that the edges with the odd label 
form a spanning tree $T$ of the graph
and that the edges with label 2 join only vertices of $T$ at even distance.
\emph{Then the derived group of $\Gamma$ 
maps onto a free group of rank  $2m$}.

To establish this claim,
fix a generator $s_0 \in S$ 
and use it to partition $S$ into two subsets $S_0$ and $S_1$ like this:
the set $S_0$ consists of all generators in $S$ at even distance from $s_0$;
here the distance of $s_0$ to another vertex $s$ is defined 
to the number of edges of the unique geodesic path on $T$ from $s_0$ to $s$. 
The set $S_1$, on the other hand, consists of all generators in $S$ 
at odd distance from $s_0$.
Note that $S_0$ and $S_1$ are disjoint and that their union is $S$. 

Let $(\bar{\Gamma}, \bar{\Delta}(\bar{S}, \bar{E}, \bar{\lambda}))$ 
be the Artin system of type $I_2(2m+1)$ 
with generating set $\bar{S} = \{\bar{s}_0, \bar{s}_1\}$.
Let $f \colon S \to \bar{S}$ be the function
which assigns to every generator $s \in S_0$ the element $\bar{s}_0$
and to every generator $s \in S_1$ the element $\bar{s}_1$,
and define $\tilde{f} \colon F(S) \epi F(\bar{S})$ to be the epimorphism of the free group on $S$ onto the free group on $\bar{S}$ that extends $f$.

The defining relations of $(\Gamma, \Delta(S, E, \lambda))$ are of two kinds.
For every pair $\{s, s'\}$ of adjacent vertices of $T$ there is a relation of the form
$ss's \cdots =s'ss'\cdots$ between alternating products of length $2m + 1$; 
the epimorphism $\tilde{f}$ maps this relation 
to a relation of $(\bar{\Gamma}, \bar{\Delta}(\bar{S}, \bar{E}, \bar{\lambda}))$.
In addition, 
there are defining relations 
which state
that certain generators $s_1$ and $s_2$ commute.
As these generators are required to be at even distance from each other,
$f$ sends these generators  to one and the same element of $\bar{S}$
and so $f(s_1)$ and $f(s_2)$ commute.
It follows that $\tilde{f}$ extends an epimorphism 
$\pi \colon \Gamma \epi \bar{\Gamma}$.
This epimorphism induces an epimorphism of $\Gamma'$ onto $\bar{\Gamma}'$.
The claim now follows from the fact 
that $\bar{\Gamma}'$ is a free group of rank $2m$
(see Proposition \ref{prp:Two-generator-Artin-groups-with-odd-label}).
\end{example}

Specimens satisfying the assumptions of Example 
\ref{example:Generalisation}
can be found among the trees displayed on pages 143--144 in \cite{Hum90},
provided one does not add to many commuting relations to the relations represented by the edges of the trees.
%

\section{Classes of groups with infinitely related metabelian tops}
\label{sec:Classes-of-groups-with-infinitely-related-metabelian-tops}
%
In this section, 
I have assembled some familiar classes of finitely generated groups
which enjoy the property 
that the metabelian top of every member in the class is infinitely related.
%
\subsection{Groups with deficiency greater than $1$}
\label{ssec:Deficiency-bigger-than-1}
I begin by recalling the notion of deficiency of a finitely presented group.
\begin{definition}
\label{definition:Deficiency-of-fp-group}
Suppose  $\Gamma$ is a group 
admitting a finite presentation,
say 
\[
\PP 
= 
\langle x_1, x_2, \ldots, x_{m(\PP)} \mid r_1, r_2, \ldots, r_{n(\PP)} \rangle
\epi \Gamma.
\]
The deficiency $\defi(\PP)$ of this presentation $\PP$ 
is the difference  $m(\PP)-n(\PP)$,
while the \emph{deficiency of} $\Gamma$ is the maximum of the differences $m(\PP)-n(\PP)$
over all finite presentations of $\Gamma$.
\end{definition}

\begin{remark}
\label{remark:Existence-deficiency}
The integer 
$
\defi(\Gamma) = \max \{\defi(\PP) \mid \PP \text{ finite presentation of }\Gamma\}
$
exists since there is an upper bound on the deficiencies of all presentations,
namely
\begin{equation}
\label{eq:Bound-deficiency}
\defi(\Gamma) \leq  \rk \Gamma_{\ab} - s(H_2(\Gamma,\Z)).
\end{equation}
Here $\rk A$ denotes the torsion-free rank of the abelian group $A$, 
while $s(A)$ is the minimal number of generators of the finitely generated abelian group $A$.
(Inequality \eqref{eq:Bound-deficiency} is due to P. Hall; 
see, \eg\cite[Lemma 1.2]{Eps61}).
\end{remark}

\subsubsection{Deficiency and infinitely related metabelianisation}
\label{sssec:Deficiency-and-infinite-relatedness-of-G/G''}
In \cite[Theorem 2]{Bau76}, 
G. Baumslag proves 
that the second homology group $H_2(\Gamma,\Z)$ of a finitely presented group is infinitely generated,
whenever $\Gamma$ admits a presentation with $m \geq 2$ generators 
and at most $m-2$ relators.
Hopf's formula
\footnote{see, \eg \cite[p.\;204]{HiSt97}}
then implies that
the metabelian top of such a group $\Gamma$ is infinitely related,
and so one has:
\begin{prp}
\label{prp:Fp-groups-with def>1}
Let $\Gamma$ be a finitely presented group with $m \geq 2$ generators and at most $m-2$ relators. Then $\Gamma/\Gamma''$  is infinitely related.
\end{prp}
%
\begin{examples}
\label{examples:Groups with deficiencygreater-than-1}
Here are some familiar classes of finitely presented groups 
with $m \geq 2$ generators and at most $m-2$ relators:
\begin{enumerate}[a)]
\item \emph{Orientable surface groups of genus $g> 1$}:
such a group has $2g \geq 4$ generators and one defining relator.
\item \emph{Non-orientable surface groups of genus $g> 2$}:
such a group has $ g\geq 3$ generators and one defining relator.
\item
\emph{Fundamental groups of connected, orientable and bounded 3-manifolds}.
The deficiency of these groups is positive 
if the manifold $M$ has a boundary component of positive genus
and it is greater than 2 
if one of its boundary components has genus greater than 1
(see, \eg Lemma V.3 in \cite{Jac80}).
\end{enumerate}
\end{examples}
%
\subsection{Groups with images of special forms}
\label{ssec:Classes-of-groups-with-special-images}
%
In this section classes of groups satisfying one or more of the hypotheses of Theorem
\ref{thm:Infinitely-related-metabelian-top}
will be exhibited.
These classes consist of groups
that are of interest to many group theorists
and which satisfy not merely hypothesis (v) 
but also one of the stronger hypotheses (i) through (iv)
listed in the statement of Theorem
\ref{thm:Infinitely-related-metabelian-top}.

\subsubsection{Orientable surface groups}
\label{sssec:Orientable-surface-groups}%
Let $\Gamma_g$ be an orientable surface group of genus $g$.
Then $\Gamma_g$ has a presentation with $2g$ generators,
say $x_1$, $y_1$, \ldots, $x_g$, $y_g$,
and the single defining relator
\begin{equation}
\label{eq:defining-relator-of-orientable-surface-group}
r_g = x_1 y_1 x_1^{-1} y_1^{-1} \cdot x_2 y_2 x_2^{-1} y_2^{-1}
\cdots x_g y_g x_g^{-1} y_g^{-1}.
\end{equation} 
This relator shows that the assignments 
\[
x_1 \mapsto t_1, \quad y_1 \mapsto 1, 
\quad x_2 \mapsto t_2, \quad y_2 \mapsto 1, \ldots
x_g \mapsto t_g, \quad y_g \mapsto 1,
\] 
extend to an epimorphism of $\Gamma_g$ onto the free group with basis 
$\{t_1, t_2, \ldots, t_g\}$.
For $g > 1$ the surface group $\Gamma_g$
 satisfies therefore hypothesis (i) listed in the statement of 
Theorem \ref{thm:Infinitely-related-metabelian-top}.

\subsubsection{Non-abelian right-angled Artin groups}
\label{sssec:Non-abelian-right-angled-Artin-groups}
Let $\Gamma$ be a right angled Artin group with $m$ generators.
Then $\Gamma$ is the quotient of the free group with basis
$\{ s_1, \ldots, s_m \}$  
by the normal subgroup $R$ generated by the commutators $[s_i, s_j]$
with $(i,j)$ ranging over a subset $ I_\Gamma $ of the set of couples
\[
J = \{(i,j) \mid 1 \leq i < j \leq m \}.
\]  
If $I_\Gamma = J$
the group  $\Gamma$ is a free abelian group of rank $m$.
Otherwise, there exists a pair of indices $i_0$ and $j_0$ with 
$(i_0,j_0) \in J\smallsetminus I_\Gamma$,
whence $\Gamma$ maps onto the free group $F$ of rank 2 
generated by the images of $s_{i_0}$ and  $s_{j_0}$
and thus satisfies hypothesis (i) listed in the statement of 
Theorem \ref{thm:Infinitely-related-metabelian-top}.
\subsubsection{Artin systems all whose edges have even labels}
\label{sssec:Artin-systems-all-whose-edges-have-even-labels}
Let $(\Gamma, \Delta(S, E, \lambda))$ be an Artin system all whose edges have even labels.
If all edges have label 2
one is in the case treated before;
assume therefore that one edge,
say  $e_0 = \{i_0, j_0\}$, 
has a label greater or equal to 4,
and let $\Gamma_{e_0}$ be the corresponding edge group.
Consider the epimorphism $\rho \colon \Gamma \epi \Gamma_{e_0}$
which sends the generators $s_{i_0}$ and $s_{j_0}$ of $\Gamma$ 
to $s_{i_0}$ and $s_{j_0}$, respectively,
and all other generators $s$ to the trivial element of $\Gamma_{e_0}$. 
Let $G$ be  the metabelian top of $\Gamma_{e_0}$.
According to Proposition  
\ref{prp:Two-generator-Artin-groups-with-even-label}
it is infinitely related
and so hypothesis (v) of Theorem \ref{thm:Infinitely-related-metabelian-top}
is satisfied.

Actually one can do better:
the group $\Gamma_{e_0}$ maps onto a wreath product of the form 
$W = \Z_m \wr C_\infty$.
As a first step in the proof of this claim,
a new presentation of $ \Gamma_{e_0}$ will be derived.
Let $F$ be the free group with basis $\{s_0,s'_0\}$ 
and let $R_0 \triangleleft F$ 
be the normal subgroup generated 
by the defining relator $r_0$ of $ \Gamma_{e_0}$; 
it has the form
\[
r_0 =(s_0s'_0)^m \cdot (s_0' s_0)^{-m}
\]
The identity $x(yx)^m x^{-1} = (xy)^m$ is valid in every group.
So the defining relator $r_0$ can also be written in the form
\[
r_0 = s_0 (s_0's_0)^m s_0^{-1} \cdot (s_0's_0)^{-m}.
\]
Set $u = s_0's_0$. 
Then $\{s_0,u\}$ is a basis of $F$
and so the group $\Gamma_{e_0}$ can also given by the presentation
\begin{equation}
\label{eq:New-presentation-for-Gamma0}
\langle s_0, u \mid s_0\cdot u^m \cdot s_0^{-1} = u^m \rangle.
\end{equation}
Consider now the wreath product $W = \Z_m \wr C_\infty$.
It is the semi-direct product $\Z_m[C_\infty] \rtimes C_\infty$
of the group ring $\Z_m[C_\infty]$ by the infinite cyclic group $C_\infty$.
Let $c$ be a generator of $C_\infty$
and let $1 \in \Z_m$ denote the unit element of the ring $\Z_m$;
it has additive order $m$.
The assignments $s _0\mapsto c \in C_\infty$ and  
$u \mapsto 1 \in \Z_m$ extend to a homomorphism $\tilde{\rho} \colon F \to W$.
Since $\tilde{\rho}(u^m) = 0$ 
this homomorphism sends the relator $r_0$ to the trivial element of $W$
and  induces therefore a homomorphism
\begin{equation}
\label{eq:Mapping-onto-W}
\rho \colon \langle s_0, u \mid s_0 u^m s_0^{-1} = u^m \rangle \to W.
\end{equation}
Its image contains the unit element $1 \in \Z_m[C_\infty]$ 
and the generator $c$ of $C_\infty$;
these two elements generate $W$,
so the homomorphism $\rho $ is surjective.
It follows that the wreath product $W$ 
is an image of the edge group $\Gamma_{e_0}$ 
and hence of $\Gamma$,
the Artin group we started with.
All taken together,
this proves that 
\emph{the Artin group $\Gamma$ satisfies hypothesis (iv) in Theorem
\ref{thm:Infinitely-related-metabelian-top}}.
%
\subsubsection{Baumslag-Solitar groups}
\label{sssec:Baumslag-Solitar-groups}
Let $m$ and $n$ be relatively prime, natural numbers greater than 1
and consider the group
\begin{equation}
\label{eq:Baumslag-Solitar-II}
\Gamma_{m,n} = \langle a, t \mid t a^m t^{-1}= a^n \rangle
\end{equation}
briefly mentioned in Remarks \ref{remarks:History-of-preceding-theorem}.
The metabelian top of $\Gamma_{m,n}$
is the split extension of the subring $\Z[1/(m \cdot n)]$ 
of the field of rational numbers $\Q$ 
by the infinite cyclic group generated by $t$.
This metabelian top is well-known to be infinitely related (\cite{BaSt76}, 
\cf\cite[Theorem C]{BiSt78}).
The group $\Gamma_{m,n}$ \emph{satisfies therefore hypothesis (v) 
listed in Theorem
\ref{thm:Infinitely-related-metabelian-top}}.

\begin{problem}
\label{question:Use-of-hypotheses-in-Theorem 3.6}
In Theorem \ref{thm:Infinitely-related-metabelian-top} five hypotheses are listed
that allow one to conclude 
that the metabelian top of a group $\Gamma$ is infinitely related.
In the examples given in the above hypothesis (i), (iii), (iv) and (v) are used,
but not hypotheses (ii).
Find familiar examples of groups where hypothesis (ii) is satisfied, 
but hypothesis (i) does not hold.
\end{problem}

%

%
\bibliography{References}
\bibliographystyle{amsalpha}%
\end{document}